\newif\iffinal
\finalfalse	

\documentclass[letterpaper,11pt,reqno]{amsart} 

\RequirePackage[utf8]{inputenc}
\usepackage[portrait,margin=3cm]{geometry}
\usepackage{color}              
\usepackage[usenames,dvipsnames]{xcolor}

\usepackage{mathrsfs,soul}
\usepackage{hyperref}
\usepackage{amssymb,amsthm,amsmath,amsfonts,amsbsy,latexsym,dsfont}
\usepackage[textsize=small]{todonotes}       
\usepackage{graphicx}
\usepackage[numeric,initials,nobysame]{amsrefs}
\usepackage{upref,setspace}

\definecolor{BrickRed}{rgb}{0.65,0.08,0}

\usepackage{enumerate}

\numberwithin{equation}{section}
\numberwithin{figure}{section}
\numberwithin{table}{section}

\newtheorem{Lemma}{Lemma}[section]
\newtheorem{Proposition}{Proposition}[section]

\newtheorem{Theorem}{Theorem}[section]

\newtheorem{Condition}{Condition}[section]

\newtheorem{Remark}{Remark}[section]
\newtheorem{Example}{Example}[section]

\newcommand{\Erdos}{Erd\H{o}s-R\'enyi }

\newcommand{\half}{{\frac{1}{2}}}

\newcommand{\lan}{\langle}
\newcommand{\ran}{\rangle}
\newcommand{\lfl}{\lfloor}
\newcommand{\rfl}{\rfloor}

\newcommand{\Rd}{{\Rmb^d}}

\newcommand{\Cmb}{{\mathbb{C}}}

\newcommand{\Emb}{{\mathbb{E}}}

\newcommand{\Nmb}{{\mathbb{N}}}

\newcommand{\Pmb}{{\mathbb{P}}}

\newcommand{\Rmb}{{\mathbb{R}}}
\newcommand{\Smb}{{\mathbb{S}}}

\newcommand{\Bmc}{{\mathcal{B}}}
\newcommand{\Cmc}{{\mathcal{C}}}

\newcommand{\Emc}{{\mathcal{E}}}
\newcommand{\Fmc}{{\mathcal{F}}}
\newcommand{\Gmc}{{\mathcal{G}}}

\newcommand{\Jmc}{{\mathcal{J}}}

\newcommand{\Lmc}{{\mathcal{L}}}
\newcommand{\Mmc}{{\mathcal{M}}}

\newcommand{\Pmc}{{\mathcal{P}}}

\newcommand{\Rmc}{{\mathcal{R}}}

\newcommand{\Tmc}{{\mathcal{T}}}

\newcommand{\one}{{\boldsymbol{1}}}

\newcommand{\mubar}{{\bar{\mu}}}

\newcommand{\btil}{{\tilde{b}}}

\newcommand{\Gtil}{{\tilde{G}}}

\newcommand{\mutil}{{\tilde{\mu}}}

\newcommand{\Tmctil}{{\tilde{\Tmc}}}

\newcommand{\xtil}{{\tilde{x}}}\newcommand{\Xtil}{{\tilde{X}}}

\def\qed{ \hfill $\blacksquare$}  

\numberwithin{equation}{section}

\begin{document}

	\title{Graphon mean field systems}

	\makeatletter
	\@namedef{subjclassname@2020}{\textup{2020} Mathematics Subject Classification}
	\makeatother

	\date{\today}
	 
	\author[Bayraktar]{Erhan Bayraktar}
		\address{Department of Mathematics, University of Michigan, 530 Church Street, Ann Arbor, MI 48109}
		\thanks{E.\ Bayraktar is partially supported by the National Science Foundation under grant DMS2106556 and by the Susan M.\ Smith chair.} 
	\author[Chakraborty]{Suman Chakraborty} 
		\address{Simons Institute of Theory of Computing, UC Berkeley, CA 94720}	
	\author[Wu]{Ruoyu Wu}
		\address{Department of Mathematics, Iowa State University, 411 Morrill Road, Ames, IA 50011} 
		\email{erhan@umich.edu, sumanc@berkeley.edu, ruoyu@iastate.edu}
			 
	\begin{abstract}
	We consider heterogeneously interacting diffusive particle systems and their large population limit.
	The interaction is of mean field type with weights characterized by an underlying graphon.
	A law of large numbers result is established as the system size increases and the underlying graphons converge.
	The limit is given by a graphon mean field system consisting of independent but heterogeneous nonlinear diffusions whose probability distributions are fully coupled.
	Well-posedness, continuity and stability of such systems are provided.
	We also consider a not-so-dense analogue of the finite particle system, obtained by percolation with vanishing rates and suitable scaling of interactions. 
	A law of large numbers result is proved for the convergence of such systems to the corresponding graphon mean field system.
	\end{abstract}	
	
	\maketitle

	\tableofcontents

\section{Introduction}
\label{sec:intro}

In this work we study mean field diffusive particle systems with heterogeneous interaction and their large population limit.
The interaction is of mean field type and is characterized through a step graphon.
More precisely, denoting by $X_i^n$ the state of the $i$-th particle,
\begin{align}
	X_i^n(t) & = X_{\frac{i}{n}}(0) + \int_0^t \frac{1}{n} \sum_{j=1}^n \xi_{ij}^n b(X_i^n(s),X_j^n(s))\,ds
	\notag \\
	& \quad + \int_0^t \frac{1}{n} \sum_{j=1}^n \xi_{ij}^n \sigma(X_i^n(s),X_j^n(s))\,dB_{\frac{i}{n}}(s), \quad i \in \{1,\dotsc,n\},
	\label{eq:dense-system-intro}
\end{align}
where $b$ and $\sigma$ are some Lipschitz functions, $\{B_u : u \in [0,1]\}$ are i.i.d.\ $d$-dimensional Brownian motions, and $X_u(0)$ is a collection of independent $\Rmb^d$-valued random variables, with probability distribution $\mu_u(0)$ for each $u \in [0,1]$, and independent of $\{B_u : u \in [0,1]\}$.
Here $\xi_{ij}^n {\in [0,1]}$ determines the interaction between particles $i$ and $j$, and depends on some step graphon $G_n$ converging to a limiting graphon in the cut metric.

The classic mean field system with homogeneous interaction, which corresponds to $\xi_{ij}^n \equiv 1$ in \eqref{eq:dense-system-intro}, dates back to works of Boltzmann, Vlasov, McKean and others (see \cite{Sznitman1991,McKean1967propagation,Kolokoltsov2010} and references therein).
While the original motivation for the study came from statistical physics, similar models have arisen in many different application areas, including economics, chemical and biological systems, communication networks and social sciences (see e.g.\ \cite{BudhirajaDupuisFischerRamanan2015limits} for an extensive list of references).
Systems with inhomogeneity described by multi-type populations, where the interaction between two particles depends on their types, have been proposed in social sciences \cite{ContucciGalloMenconi2008phase}, statistical mechanics \cite{Collet2014macroscopic}, neurosciences \cite{BaladronFaugeras2012}, and others \cite{BudhirajaWu2016some,NadtochiyShkolnikov2018mean}. 
In recent years, there have been an increasing attention on the study of mean field systems on large networks, including {\cite{Delattre2016,BhamidiBudhirajaWu2019weakly,Delarue2017mean,CoppiniDietertGiacomin2019law,OliveiraReis2019interacting,BudhirajaMukherjeeWu2019supermarket,BayraktarWu2019mean,Lucon2020quenched,Medvedev2014nonlinear1,Medvedev2014nonlinear2,KaliuzhnyiMedvedev2018mean,DupuisMedvedev2020large,Coppini2019long,BetCoppiniNardi2020weakly}}, where the majority of focus is on \Erdos {type} random graphs.
Among these, \cite{OliveiraReis2019interacting} allows the edge probability between two nodes to depend on independent random media variables associated with these two nodes, and \cite{Delarue2017mean} analyzes the mean field game on \Erdos random graphs. 

We extend the study of mean field models to a much larger class of graphs. 
To put our work in context, let us describe the class of graphs that we are going to consider. 
We consider sequences of dense graphs that converge to a limit in an appropriate sense (see \cite{Lovasz2012large} and references therein). 
Roughly speaking, this limit theory treats a graph $\mathbb{G}_n$ on $n$ vertices as a function ${G_n} \colon I \times I \to \Rmb$, where $I:=[0,1]$. 
This function is what we call a graphon.
Then $\mathbb{G}_n$ is said to converge to the function $G$ if and only if ${G_n}$ converges to $G$ in ``cut metric" (see Section \ref{sec:limit-system} for the definition). 
We consider mean field models on such converging graph sequences. 
The motivation for considering such graph sequences is that aside from its theoretical implications, many important graph models (both random and deterministic) have been shown to converge to a limit. 
See \cite{chatterjee2013estimating,bhamidi2018weighted,basak2016large,krivelevich2006pseudo} for many such examples. 
Unfortunately, the graph limit theory only works for dense graphs (graphs with order of $n^2$ many edges). 
To extend our study to the not-so-dense settings we also consider bond percolated models of graphs; see \cite{bollobas2010percolation}, {and also see \cite{BorgsChayesCohnZhao2019Lp,BorgsChayesCohnZhao2018Lp} and references therein for more interests in and analysis on not-so-dense graphs}.

Among the aforementioned works on mean field systems on large networks, the ones closest to our setup are \cite{Lucon2020quenched} and \cite{BetCoppiniNardi2020weakly}.
The work \cite{Lucon2020quenched} considers large population of diffusions interacting on a random directed inhomogeneous graph, and obtains quenched convergence of particles, empirical measures, and also spatial fields (empirical measures of particle states and positions/indices).
Compared with our models and results, there are two main differences.
Firstly, \cite{Lucon2020quenched} considers independent Bernoulli edges scaled by a certain dilution parameter and assumes the collection of scaled success probabilities, viewed as graphons/kernels, converges in a sense stronger than the usual cut metric of graphons (Assumption 2.9), with Lipschitz assumptions on initial states (Assumption 2.6) and suitable regularity assumptions on the graphons (Assumptions 2.11 and 2.14).
While we consider both Bernoulli edges and $[0,1]$-valued weighted edges, converging in the cut metric, with continuity assumptions on initial states.
Secondly, \cite{Lucon2020quenched} obtains quenched convergence, conditioning on the realization of random graphs; while we are mostly concerned with the annealed setup and some of our results (in Theorems \ref{thm:dense-LLN} and \ref{thm:sparse-LLN}) could be applied to the quenched setup (see Condition \ref{cond:dense-Gn}(b) and Remark \ref{rmk:3.3}(b)).

After the initial submission of this work, another paper on graphon particle systems (\cite{BetCoppiniNardi2020weakly}) was posted on arXiv a few months later.
\cite{BetCoppiniNardi2020weakly} considers weakly interacting oscillators on dense random graphs and obtains an annealed law of large numbers for the empirical measure.
The limiting system is given by describing the evolution of a particle uniformly chosen from $I$.
{
However, \cite{BetCoppiniNardi2020weakly} does not consider not-so-dense graphs as in our Section \ref{sec:sparse-system}, and their results do not imply our results in Section \ref{sec:dense-system} (and ours do not imply theirs either).
In particular, \cite{BetCoppiniNardi2020weakly} assumes $\{0,1\}$-valued random adjacency matrix $\xi^n$ for $n$-particle systems converging to a possibly random graphon in the cut distance (and the argument is expected to work for $[0,1]$-valued weighted edges as well)}, while we consider both such random $\xi^n$ and deterministic $[0,1]$-valued $\xi^n$.
Also, \cite{BetCoppiniNardi2020weakly} assumes homogeneous initial state distribution and constant diffusion coefficient, while we allow for heterogeneous initial state distribution and state-and-interaction-dependent diffusion coefficient.
Lastly, particles' state space is bounded (one dimensional torus) in \cite{BetCoppiniNardi2020weakly}, and hence the system coefficients are bounded, while our work allows coefficients to have linear growth.

We should point out that the use of graphons to analyze heterogeneous interaction in game theory emerged recently (see e.g.\ \cite{CainesHuang2018graphon,PariseOzdaglar2019graphon,Carmona2019stochastic}). 
Among these, \cite{PariseOzdaglar2019graphon} analyzes static graphon games and the convergence of the $n$-player game with interaction network sampled from a given graphon.
Static graphon games are considered in \cite{Carmona2019stochastic} and the convergence of the $n$-player game with general interaction network that converges to a given graphon is obtained.
The diffusive dynamics for the states of the particles, with constant diffusion coefficients, is considered in \cite{CainesHuang2018graphon} for continuum graphon mean field games.
However, \cite{CainesHuang2018graphon} does not address the convergence problem of the finite particle system to the limiting problem they analyze.

The goal of this work is to study the asymptotics of the diffusive particle system \eqref{eq:dense-system-intro} with heterogeneous interaction, and their not-so-dense analogue in \eqref{eq:sparse-system-intro} below. 
Our first main result is the existence, uniqueness, continuity, and stability property (Proposition \ref{prop:limit-existence-uniqueness} and Theorem \ref{thm:limit-system-property}) for the limiting graphon particle system \eqref{eq:limit-system-dense}, consisting of a continuum of independent but non-identical nonlinear diffusions. 
Among these, the stability property (Theorem \ref{thm:limit-system-property}(c)) in particular says that the system solution converges in a suitable sense provided that the underlying graphon converges in the cut metric.
The proof makes use of a coupling argument but challenge is two fold:
First, the cut metric is in general very weak, in that 
the convergence $G_n \to G$ does not necessarily imply the $L^2$-convergence of $G_n$ as operators on $I \times I$, namely one may not have $\int_{I \times I} [G_n(x,y)-G(x,y)]^2 \,dx\,dy \to 0$.
However, one could alternatively view $G_n$ as operators from $L^\infty(I)$ to $L^1(I)$ that are continuous with respect to the cut metric (see Remark \ref{rmk:graphon-convergence}).
This observation is actually an important building block of many proofs in this work.
Second, the interaction in the graphon particle system \eqref{eq:limit-system-dense} does not match with such a choice of operator, unless the coefficients $b(x,y)$ and $\sigma(x,y)$ could be decomposed as the product of functions of each variable.
For this, a truncation and approximation argument is applied to these coefficients, and the associated errors are carefully analyzed (see Section \ref{sec:limit-system-property-pf}).

The second main result is the convergence of the $n$ particle system \eqref{eq:dense-system-intro} to the graphon particle system \eqref{eq:limit-system-dense}, for a sequence of convergent underlying step graphons (graphons with blockwise constant values; see \eqref{eq:step-graphon}).
A law of large numbers (LLN) is established in Theorem \ref{thm:dense-LLN}, which says that the empirical measure of $n$ particles in \eqref{eq:dense-system-intro} converges in probability to the averaged distribution of a continuum particles in \eqref{eq:limit-system-dense}. 
This is first proved in Lemma \ref{lem:new1} under certain regularity assumptions of the graphon (Condition \ref{cond:limit-G}(b)), by applying again a truncation and approximation argument to the system coefficients.
Then the stability property (Theorem \ref{thm:limit-system-property}(c)) is used in Section \ref{sec:dense-LLN-pf} to show that the result also holds for general graphons.
In Theorem \ref{thm:dense-LLN-special}, we also obtain a precise particle-wise uniform convergence rate, when the underlying step graphons are sampled from a given graphon with a certain continuity property.

Our third main result is the analysis of the not-so-dense analogue of \eqref{eq:dense-system-intro}:
\begin{align}
	X_i^n(t) & = X_{\frac{i}{n}}(0) + \int_0^t \frac{1}{n\beta_n} \sum_{j=1}^n \xi_{ij}^n b(X_i^n(s),X_j^n(s))\,ds + \int_0^t \sigma(X_i^n(s))\,dB_{\frac{i}{n}}(s), \quad i \in \{1,\dotsc,n\}, 
	\label{eq:sparse-system-intro}
\end{align}
where $\beta_n \in (0,1]$ is some sequence of numbers that may converge to $0$, and $\{\xi_{ij}^n\}$ are independent Bernoulli random variables with possibly vanishing probabilities (of order $\beta_n$) associated with the underlying step graphon.
Similar to mean field systems on \Erdos random graphs \cite{BhamidiBudhirajaWu2019weakly,Delarue2017mean,OliveiraReis2019interacting}, the strength of interaction here is scaled by the order of the number of neighbors
(see Remark \ref{rmk:sparse} for more explanations).
In Theorem \ref{thm:sparse-LLN}, we show a LLN that the limit of such systems is again given by a graphon particle system, provided that the underlying step graphons converge and $$\lim_{n \to \infty} n\beta_n = \infty.$$ 
The main challenge lies in the heterogeneity of the system and the average of interactions of order $\xi_{ij}^n / \beta_n$ that is unbounded in $n$.
The unbounded interaction $\xi_{ij}^n / \beta_n$ is taken care of in \cite{BhamidiBudhirajaWu2019weakly,Delarue2017mean,OliveiraReis2019interacting} using exchangeability.
However, due to the lack of exchangeability here, a new approach is needed.
Indeed, besides the application of coupling, truncation and approximation arguments, the key ingredient in the proof of Theorem \ref{thm:sparse-LLN} is \eqref{eq:sparse-R2} in Lemma \ref{lem:new-sparse}, which shows that the expected effect of unbounded interactions $\xi_{ij}^n/\beta_n$ on the coupled difference $|X_j^n-X_{\frac{j}{n}}|$ is roughly the same as $\Emb[\xi_{ij}^n/\beta_n] \Emb|X_j^n-X_{\frac{j}{n}}|$, up to some constant multiples and negligible errors.
The proof of Lemma \ref{lem:new-sparse} applies a collection of change of measure arguments separately to each pair (resp.\ triplet) of certain auxiliary particles and the edge (resp.\ edges) connecting them.
For each change of measure, the Radon-Nikodym derivative and the difference among pre-limit, limiting and auxiliary particles are carefully analyzed.
Due to the technical application of the Girsanov's Theorem, the diffusion coefficient in \eqref{eq:sparse-system-intro} is taken to be state-dependent only.
Lastly, we also obtain a precise rate of convergence in Theorem \ref{thm:sparse-LLN-special}, when the underlying step graphons are sampled from a given graphon with a certain continuity property.

\subsection{Organization}
The paper is organized as follows.
In Section \ref{sec:limit-system} we analyze the graphon particle system \eqref{eq:limit-system-dense}.
The existence and uniqueness is proved in Proposition \ref{prop:limit-existence-uniqueness}.
The continuity and stability of the system is presented in Theorem \ref{thm:limit-system-property} in Section \ref{sec:limit-system-property}.
Concrete examples are given in Section \ref{sec:limit-example}.
In Section \ref{sec:dense-system} we study the convergence of the $n$ particle system \eqref{eq:dense-system-intro}.
The LLN is given in Theorem \ref{thm:dense-LLN}, and a precise rate of convergence is given in Theorem \ref{thm:dense-LLN-special} under conditions.
In Section \ref{sec:sparse-system} we study the convergence of the $n$ particle system with not-so-dense interaction \eqref{eq:sparse-system-intro}.
The LLN is given in Theorem \ref{thm:sparse-LLN}, and a precise rate of convergence is given in Theorem \ref{thm:sparse-LLN-special} under conditions.
Sections \ref{sec:limit-system-pf}, \ref{sec:dense-system-pf} and \ref{sec:sparse-system-pf} are devoted to the proofs of results in Section \ref{sec:limit-system}, \ref{sec:dense-system} and \ref{sec:sparse-system}, respectively.

We close this section by introducing some frequently used notation.

\subsection{Notation} 
\label{sec:notation}

Given a Polish space $\Smb$, denote by $\Pmc(\Smb)$ the space of probability measures on $\Smb$ endowed with the topology of weak convergence.
For $\mu \in \Pmc(\Smb)$ and a $\mu$-integrable function $f \colon \Smb \to \Rmb$, let $\lan f,\mu \ran := \int_\Smb f(x)\,\mu(dx)$.
For $f \colon \Smb \to \Rmb$, let $\|f\|_\infty := \sup_{x \in \Smb} |f(x)|$.
The probability law of a random variable $X$ will be denoted by $\Lmc(X)$. 
Fix $T \in (0,\infty)$ and all processes will be considered over the time horizon $[0,T]$.
Denote by $\Cmb([0,T]:\Smb)$ the space of continuous functions from $[0, T]$ to $\Smb$, endowed with the topology of uniform convergence.
Let $\Cmc_d := \Cmb([0,T]:\Rd)$ and $\|x\|_{*,t} := \sup_{0 \le s \le t} |x_s|$ for $x \in \Cmc_d$ and $t \in [0,T]$.
We will use $\kappa$ to denote various constants in the paper and $\kappa(m)$ to emphasize the dependence on some parameter $m$.
Their values may change from line to line.
Expectations under $\Pmb$ will be denoted by $\Emb$.
To simplify the notation, we will usually write $\Emb[ X^2 ]$ as $\Emb X^2$.

\section{Graphon particle systems}
\label{sec:limit-system}

We follow the notation used in \cite[Chapters 7 and 8]{Lovasz2012large}.
Let $I := [0,1]$.
Denote by $\Gmc$ the space of all bounded symmetric measurable functions $G \colon I \times I \to \Rmb$.
A graphon $G$ is an element of $\Gmc$ with $0 \le G \le 1$.
The cut norm on $\Gmc$ is defined by
$$\|G\|_\square := \sup_{S,T \in \Bmc(I)} \left| \int_{S \times T} G(u,v)\,du\,dv \right|,$$
and the corresponding cut metric is defined by 
$$d_\square(G_1,G_2) := \|G_1-G_2\|_\square.$$

\begin{Remark}
	\label{rmk:graphon-convergence}
	We will also view a graphon $G$ as an operator from $L^\infty(I)$ to $L^1(I)$ with the operator norm
	$$\|G\|:=\|G\|_{\infty\to1}:=\sup_{\|g\|_\infty \le 1} \|Gg\|_1 = \sup_{\|g\|_\infty \le 1} \int_I \left| \int_I G(u,v) g(v) \, dv \right| du.$$	
	From \cite[Lemma 8.11]{Lovasz2012large} it follows that if $\|G_n-G\|_\square \to 0$ for a sequence of graphons $G_n$, then $\|G_n-G\| \to 0$.
\end{Remark}

Given a graphon $G$ and an initial distribution $\mu(0) := (\mu_u(0) \in \Pmc(\Rmb^d) : u \in I)$, consider the following graphon particle system:
\begin{align}
	X_u(t) & = X_u(0) + \int_0^t \int_I \int_{\Rmb^d} b(X_u(s),x)G(u,v)\,\mu_{v,s}(dx)\,dv\,ds \notag \\
	& \quad + \int_0^t \int_I \int_{\Rmb^d} \sigma(X_u(s),x)G(u,v)\,\mu_{v,s}(dx)\,dv\,dB_u(s), \:\: \mu_{u,t}=\Lmc(X_u(t)), \:\: u \in I. \label{eq:limit-system-dense}
\end{align}
As introduced in Section \ref{sec:intro}, here $\{B_u : u \in I\}$ are i.i.d.\ $d$-dimensional Brownian motions, and $X_u(0)$ is a collection of independent $\Rmb^d$-valued random variables, with law $\mu_u(0)$ for each $u \in I$, and independent of $\{B_u : u \in I\}$, defined on some filtered probability space $(\Omega,\Fmc,\Pmb,\{\Fmc_t\})$.
We make the following assumptions on the initial states and system coefficients.

\begin{Condition}
	\phantomsection
	\label{cond:standing}
	\begin{enumerate}[(a)]
	\item The map $I \ni u \mapsto \mu_u(0)=\Lmc(X_u(0)) \in \Pmc(\Rmb^d)$ is measurable. There exists some $\varepsilon \in (0,\infty)$ such that
	\begin{equation}
		\label{eq:initial-moment}
		\sup_{u \in I} \Emb |X_u(0)|^{2+\varepsilon} < \infty.
	\end{equation}
	\item The coefficients $b$ and $\sigma$ are Lipschitz functions, namely there exists some $K \in [0,\infty)$ such that
	$$|b(x,y)-b(x',y')| + |\sigma(x,y)-\sigma(x',y')| \le K(|x-x'|+|y-y'|), \quad \forall x,x',y,y' \in \Rd.$$ 
	\end{enumerate}
\end{Condition}

\begin{Remark}
	The assumption that $\varepsilon>0$ is used for some technical arguments treating the unboundedness of $b$ and $\sigma$.
	If $b$ and $\sigma$ are bounded and Lipschitz functions, then one could take $\varepsilon=0$ (see Remarks \ref{rmk:epsilon-1} and \ref{rmk:epsilon-2}).
\end{Remark}

The following proposition gives well-posedness for the system \eqref{eq:limit-system-dense}. 
For graphon mean field games with diffusive dynamics, under certain regularity conditions, existence and uniqueness were shown in \cite[Theorem 2]{CainesHuang2018graphon} (and the journal version \cite[Theorem 3.12]{CainesHuang2020graphon}).
But we provide a proof in Section \ref{sec:limit-existence-uniqueness-pf} as our graphon $G$ may not be continuous.

\begin{Proposition}
	\label{prop:limit-existence-uniqueness}
	Suppose Condition \ref{cond:standing} holds. Then
	there exists a unique pathwise solution $\{X_u\}$ to the graphon particle system \eqref{eq:limit-system-dense}. Moreover, $\sup_{u \in I} \Emb\|X_u\|_{*,T}^{2+\varepsilon}<\infty$ and the map $I \ni u \mapsto \mu_u \in \Pmc(\Cmc_d)$ is measurable.
\end{Proposition}

\begin{Remark}
	\begin{enumerate}[(a)]
	\item 
		We note that processes $\{X_u\}$ in \eqref{eq:limit-system-dense} are independent but not identically distributed nonlinear diffusions.
		In particular, the Fokker--Planck equations for {the probability distributions of} $\{X_u\}$ are nonlinear and fully coupled.
		In general, each $X_u$ may not be a McKean--Vlasov process, as the probability law $\mu_u$ plays a negligible role in its evolution. 
	\item
		{We also note that we are not assuming $u \mapsto B_u$ is measurable or claiming $u \mapsto X_u$ is measurable. Proposition \ref{prop:limit-existence-uniqueness} gives a measurable dependence of the \textbf{law} $\mu_u$ on $u$, which is sufficient for later use since \eqref{eq:limit-system-dense} involves integrals of $\mu_u$, instead of $X_u$ or $B_u$, with respect to $u \in [0,1]$.}
	\end{enumerate}	
\end{Remark}

In order to analyze the collection of probability laws $\mu = (\mu_u : u \in I)$, consider the following space of probability measures
\begin{align*}
	\Mmc & := \{ \nu = (\nu_u : u \in I) \in [\Pmc(\Cmc_d)]^I \,|\, I \ni u \mapsto \nu_u \in \Pmc(\Cmc_d) \mbox{ is measurable and } \\
	& \qquad\qquad\qquad\qquad\qquad\qquad\qquad \sup_{u \in I} \int_{\Cmc_d} \|x\|_{*,T}^2 \,\nu_u(dx) < \infty \}.
\end{align*}
For the convenience of analysis (see e.g.\ Remark \ref{rmk:Wass-dual}), we make use of the following Wasserstein-$2$ metrics
\begin{align}
	W_2(\mu,\nu) & := \left( \inf \left\{ \Emb |X-\Xtil|^2 : \Lmc(X)=\mu,\Lmc(\Xtil)=\nu \right\} \right)^{1/2}, \: \mu,\nu \in \Pmc(\Rmb^d), \label{eq:Wasserstein-R} \\
	W_{2,t}(\mu,\nu) & := \left( \inf \left\{ \Emb \|X-\Xtil\|_{*,t}^2 : \Lmc(X)=\mu,\Lmc(\Xtil)=\nu \right\} \right)^{1/2}, \: t \in [0,T], \: \mu,\nu \in \Pmc(\Cmc_d), \label{eq:Wasserstein} \\
	W_{2,t}^\Mmc(\mu,\nu) & := \sup_{u \in I} W_{2,t}(\mu_u,\nu_u), \: t \in [0,T], \: \mu,\nu \in \Mmc. \label{eq:Wasserstein-M}
\end{align}

\begin{Remark}
	\label{rmk:Wass-dual}
	From \eqref{eq:Wasserstein-R}, \eqref{eq:Wasserstein} and \eqref{eq:Wasserstein-M} clearly we have
	\begin{align*}
		W_2(\mu,\nu) & \ge \sup_{f} \left| \int_{\Rd} f(x) \, \mu(dx) - \int_{\Rd} f(x) \, \nu(dx) \right|, \quad \mu,\nu \in \Pmc(\Rmb^d), \\
		W_{2,t}(\mu,\nu) & \ge \sup_{f} \left| \int_{\Rd} f(x) \, \mu_t(dx) - \int_{\Rd} f(x) \, \nu_t(dx) \right|, \quad \mu,\nu \in \Pmc(\Cmc_d), \\
		W_{2,t}^\Mmc(\mu,\nu) & \ge \sup_{u \in I} \sup_{f} \left| \int_{\Rd} f(x) \, \mu_{u,t}(dx) - \int_{\Rd} f(x) \, \nu_{u,t}(dx) \right|, \quad \mu,\nu \in \Mmc, 	
	\end{align*}
	for each $t \in [0,T]$,	where the supremum is taken over all $f \colon \Rmb^d \to \Rmb$ such that the integrals exist and $|f(x)-f(\xtil)| \le |x-\xtil|$ for $x,\xtil \in \Rd$.
\end{Remark}

\subsection{Continuity and stability of the system}
\label{sec:limit-system-property}

In this section we are interested in establishing the continuity and stability properties for the graphon particle system \eqref{eq:limit-system-dense}.
We usually make the following assumption on the initial distribution $\mu(0)$ and the graphon $G$.

\begin{Condition}
	\phantomsection
	\label{cond:limit-G}
	There exists a finite collection of intervals $\{I_i : i=1,\dotsc,N\}$ for some $N \in \Nmb$, such that $\cup_{i=1}^N I_i = I$ and for each $i \in \{1,\dotsc,N\}$:	
	\begin{enumerate}[(a)]
	\item The map $I_i \ni u \mapsto \mu_u(0) \in \Pmc(\Rd)$ is continuous with respect to the $W_2$ metric.
	
	\item For each interior point $u \in I_i$, there exists a subset $A_u \subset I$ such that $\lambda_I(A_u)=0$ and $G(u,v)$ is continuous at $(u,v) \in I \times I$ for each $v \in I \setminus A_u$, where $\lambda_I$ denotes the Lebesgue measure on $I$.
	\end{enumerate}
\end{Condition}

\begin{Remark}
	Condition \ref{cond:limit-G}(b) is not necessary for, but will help strengthen, convergence results in Sections \ref{sec:dense-system} and \ref{sec:sparse-system}.
	It holds naturally if $G$ is continuous, or if $G$ is continuous when restricted to each block $I_i \times I_j$.	
	For example, it holds for graphons
	$G(u,v)=\one_{[0,\frac{1}{2}]^2}(u,v)$ and 
	$G(u,v)=\one_{\{|u-v| \le \frac{1}{4}\}}(u,v)$.
\end{Remark}

Sometimes we may work with a special class of $\mu(0)$ and $G$ having certain Lipschitz properties.

\begin{Condition}
	\label{cond:limit-G-special}
	There exists some $\kappa \in (0,\infty)$ and a finite collection of intervals $\{I_i : i=1,\dotsc,N\}$ for some $N \in \Nmb$, such that $\cup_{i=1}^N I_i = I$ and
	\begin{align*}
		W_2(\mu_{u_1}(0),\mu_{u_2}(0)) & \le \kappa |u_1-u_2|, \quad u_1,u_2 \in I_i, \quad i \in \{1,\dotsc,N\}, \\
		|G(u_1,v_1)-G(u_2,v_2)| & \le \kappa (|u_1-u_2|+|v_1-v_2|), \: (u_1,v_1),(u_2,v_2) \in I_i \times I_j, \: i,j \in \{1,\dotsc,N\}.
	\end{align*}
\end{Condition}

The following theorem gives continuity and stability of the graphon particle system \eqref{eq:limit-system-dense}.
The proof is given in Section \ref{sec:limit-system-property-pf}.

\begin{Theorem}
	\phantomsection
	\label{thm:limit-system-property}
	Suppose Condition \ref{cond:standing} holds.
	\begin{enumerate}[(a)]
	\item (Continuity)
		Suppose Condition \ref{cond:limit-G} holds. Then for each $i \in \{1,\dotsc,N\}$, the map $I_i \ni u \mapsto \mu_u \in \Pmc(\Cmc_d)$ is continuous with respect to the $W_{2,T}$ metric.
	\item (Lipschitz continuity)
		Suppose Condition \ref{cond:limit-G-special} holds. Then there exists some $\kappa \in (0,\infty)$ such that $W_{2,T}(\mu_u,\mu_v) \le \kappa|u-v|$ whenever $u,v \in I_i$ for some $i \in \{1,\dotsc,N\}$.
	\item (Stability)
		Let $\mu^G$ be the probability law of \eqref{eq:limit-system-dense} associated with $G$. 
		The map $G \mapsto \mu^G$ is continuous in the sense that $\int_I [W_{2,T}(\mu^{G_n}_u,\mu^G_u)]^2 \, du \to 0$ if a sequence of graphons $G_n \to G$ in the cut metric. 
	\end{enumerate}	
\end{Theorem}

\begin{Remark}
	\begin{enumerate}[(a)]
	\item Theorem \ref{thm:limit-system-property} (a,b) will be needed in Sections \ref{sec:dense-system} and \ref{sec:sparse-system} to analyze the convergence of $n$-particle systems with graphon mean field interactions.
	
	\item Theorem \ref{thm:limit-system-property}(c) implies that the solution law to \eqref{eq:limit-system-dense} depends on the underlying graphon $G$ in a continuous manner. This and Proposition \ref{prop:limit-existence-uniqueness} together guarantees that the analysis of \eqref{eq:limit-system-dense} is ``well-posed" according to Hadamard's principle (cf.\ \cite[Page 368]{Beeson2012foundations} and \cite[Page 38]{Hadamard2003lectures}).
	It will also be used in Sections \ref{sec:dense-system} and \ref{sec:sparse-system} to analyze the convergence of $n$-particle systems when $G$ is not necessarily continuous.
	\end{enumerate}
\end{Remark}

\subsection{Some special graphon particle systems}
\label{sec:limit-example}

In this section we introduce two special scenarios under which the system \eqref{eq:limit-system-dense} is more tractable. 

\begin{Example}
Suppose $G$ is blockwise constant (which arises as a limit of the stochastic block model), that is, there exists a finite collection of intervals $\{I_i : i=1,\dotsc,N\}$ and constants $\{p_{ij}=p_{ji} \in [0,1] : i,j = 1,\dotsc,N\}$ for some $N \in \Nmb$, such that $\cup_{i=1}^N I_i = I$ and
\begin{equation*}
	G(u,v) = p_{ij}, \quad (u,v) \in I_i \times I_j, \quad i,j \in \{1,\dotsc,N\}.
\end{equation*}
Also suppose the initial distribution is the same within each interval:
$$\mu_u(0)=\mu_v(0), \quad u,v \in I_i, \quad i \in \{1,\dotsc,N\}.$$
Due to the homogeneity in this case, the system \eqref{eq:limit-system-dense} could be written in terms of just $N$ representatives $u_i \in I_i$, $i \in \{1,\dotsc,N\}$:
\begin{align*}
	X_{u_i}(t) & = X_{u_i}(0) + \int_0^t \sum_{j=1}^N |I_j| p_{ij} \left( \int_{\Rmb^d} b(X_{u_i}(s),x)\,\mu_{u_j,s}(dx) \right) ds \\
	& \quad + \int_0^t \sum_{j=1}^N |I_j| p_{ij} \left( \int_{\Rmb^d} \sigma(X_{u_i}(s),x)\,\mu_{u_j,s}(dx) \right) dB_{u_i}(s), \quad \mu_{u_i,t}=\Lmc(X_{u_i}(t)),
\end{align*}
where $|A|$ denotes the Lebesgue measure of $A \subset I$.
Note that this is simply a finite collection of multi-type McKean--Vlasov processes.
\end{Example}

\begin{Example}
	Suppose $b(x,y)=c_1+c_2x+c_3y$ is linear, $\sigma$ is constant, and the initial distributions $\{\mu_u(0) : u \in I\}$ are Gaussian.
	Then the system \eqref{eq:limit-system-dense} is just consisting of a collection of Gaussian processes.
	Letting $m_u(t):=\Emb[X_u(t)]$ and $M_u(t):=\Emb[X_u^2(t)]$, we have
	\begin{align*}
		m_u(t) & = m_u(0) + \int_0^t \int_0^1 (c_1+c_2m_u(s)+c_3m_v(s))G(u,v) \,dv\,ds, \\
		M_u(t) & = M_u(0) + \Emb\left[\int_0^t 2X_u(s)\,dX_u(s)\right] + \sigma^2\left(\int_0^1G(u,v)\,dv\right)^2t \\
		& = M_u(0) + 2\int_0^t \int_0^1 (c_1m_u(s)+c_2M_u(s)+c_3m_u(s)m_v(s))G(u,v)\,dv\,ds \\
		& \qquad + \left(\sigma\int_0^1G(u,v)\,dv\right)^2t. 
	\end{align*}
	This is a system of coupled ordinary differential equations.
\end{Example}

\begin{Remark}
	We note that even in the above two examples, it may not be an easy work to obtain explicit forms of solutions $\Lmc(X_u(t))$.
	In the next two sections we will show the convergence of finite particle systems \eqref{eq:dense-system-intro} and \eqref{eq:sparse-system-intro} to the graphon particle system \eqref{eq:limit-system-dense}.
	On one hand, such results could be used to approximate large finite particle systems with heterogeneous interactions by a graphon particle system.
	On the other hand, for graphon particle systems that are not tractable, one may choose a suitable finite particle system (and even its Euler discretizations) to approximate the former.
\end{Remark}

\section{Mean-field particle systems on dense graphs}
\label{sec:dense-system}

In this section, we consider a sequence of $n$ interacting diffusions \eqref{eq:dense-system-intro} with the strength of interaction governed by $\xi_{ij}^n$ associated with some kernel $G_n$:
\begin{align}
	X_i^n(t) & = X_{\frac{i}{n}}(0) + \int_0^t \frac{1}{n} \sum_{j=1}^n \xi_{ij}^n b(X_i^n(s),X_j^n(s))\,ds \notag \\
	& \quad + \int_0^t \frac{1}{n} \sum_{j=1}^n \xi_{ij}^n \sigma(X_i^n(s),X_j^n(s))\,dB_{\frac{i}{n}}(s), \quad i \in \{1,\dotsc,n\}. 
	\label{eq:dense-system}
\end{align}
Here the pathwise existence and uniqueness of the solution is guaranteed by the Lipschitz property of $b$ and $\sigma$.

{We would like to consider the natural correspondence between the adjacency matrix $\{\xi_{ij}^n\}$ and a function on $I \times I$ with constant value $\xi_{ij}^n$ on each block $(\frac{i-1}{n},\frac{i}{n}] \times (\frac{j-1}{n},\frac{j}{n}]$ of side length $\frac{1}{n}$.
So we} make the following assumption on the strength of interaction $\xi_{ij}^n$ and the associated kernel $G_n$.

\begin{Condition}
	\phantomsection
	\label{cond:dense-Gn}
	\begin{enumerate}[(a)]
	\item 
		{The kernel} $G_n$ is a step graphon, 
		that is, $0 \le G_n \le 1$ and
		\begin{equation}
			\label{eq:step-graphon}
			G_n(u,v) = G_n \left( \frac{\lceil nu \rceil}{n}, \frac{\lceil nv \rceil}{n} \right), \quad \mbox{ for } (u,v) \in I \times I.
		\end{equation}	
	\item
		{The strength of interaction $\xi_{ij}^n$ is generated by the kernel $G_n$ as follows. 
		Either one of the following two cases holds for all $\{\xi_{ij}^n : i,j \in \{1,\dotsc,n\}, n \in \Nmb\}$:
		\begin{enumerate}[(b.1)]
		\item $\xi_{ij}^n=G_n(\frac{i}{n},\frac{j}{n})$;
		\item $\xi_{ij}^n=\xi_{ji}^n=$Bernoulli$(G_n(\frac{i}{n},\frac{j}{n}))$ independently for $1 \le i \le j \le n$, and independent of $\{B_u, X_u(0) : u \in I\}$.
		\end{enumerate}}
	\item
		$G_n \to G$ in the cut metric as $n \to \infty$.
	\end{enumerate}
\end{Condition}

\begin{Remark}
	\label{rmk:3.1}
	Although in this work we consider $G_n \to G$ in the cut metric, it is also possible to analyze the convergence to the limiting graphon particle system \eqref{eq:limit-system-dense} when $G_n \to G$ in the cut distance, with suitable additional assumptions, illustrated as follows.
	
	The cut distance on $\Gmc$ is defined by 
	$$\delta_\square(G_1,G_2) := \inf_{\varphi \in S_I} \|G_1-G_2^\varphi\|_\square,$$
	where $S_I$ denotes the set of all invertible measure preserving maps $I \to I$, and $G^\varphi(u,v):=G(\varphi(u),\varphi(v))$.		
	If $\delta_\square(G_n,G) \to 0$ for a sequence of step graphons $G_n$, then it follows from \cite[Theorem 11.59]{Lovasz2012large} that $\|\widetilde{G}_n-G\|_\square \to 0$, where $\widetilde{G}_n$ is a suitable relabeling of $G_n$.
	That is, there exists some permutation $\tilde{\varphi}_n$ on $\{1,\dotsc,n\}$ such that, letting $\varphi_n \in S_I$ be
	$$\varphi_n(\frac{i}{n}-u) = \frac{\tilde{\varphi}_n(i)}{n}-u, \quad 0 \le u < \frac{1}{n}, \quad i \in \{1,\dotsc,n\},$$ and $\widetilde{G}_n = G_n^{\varphi_n}$, one has $d_\square(\widetilde{G}_n,G) \to 0$.
	
	Therefore Condition \ref{cond:dense-Gn}(c) could be replaced by the following one, in terms of the cut distance with an additional assumption to guarantee that the initial conditions are matching (which holds naturally if $\{X_i^n(0)\}$ are i.i.d.):
	\begin{itemize}
	\item[(c*)] $\delta_\square(G_n,G) \to 0$ as $n \to \infty$, so that there exists some invertible measure preserving map $\varphi_n \colon I \to I$, interpreted as a relabeling of the graph $G_n$, such that $d_{\square}(G_n^{\varphi_n},G) \to 0$. Also suppose $X_i^n(0) = X_{\varphi_n^{-1}(\frac{i}{n})}(0)$.
	\end{itemize}
	One could then apply results in this work to $Y_i^n := X_{n\varphi_n(\frac{i}{n})}^n$ with the kernel $\widetilde{G}_n :=G_n^{\varphi_n}$, on noting that $Y_i^n(0) = X_{n\varphi_n(\frac{i}{n})}^n(0) = X_{\frac{i}{n}}(0)$.
	This would also imply, as one would naturally expect, the convergence of empirical measures is independent of the labeling of the particles; see e.g.\ \cite{BetCoppiniNardi2020weakly} for more discussions on this.
\end{Remark}

The following convergence holds for the system \eqref{eq:dense-system}.
The proof is given in Section \ref{sec:dense-LLN-pf}.

\begin{Theorem}
	\label{thm:dense-LLN}
	Suppose Conditions \ref{cond:standing}, \ref{cond:limit-G}(a) and \ref{cond:dense-Gn} hold.
	Then
	\begin{equation}
		\label{eq:dense-mu-cvg}
		\mu^n \to \mubar \text{ in } \Pmc(\Cmc_d) \text{ in probability}
	\end{equation}
	as $n \to \infty$,
	where
	\begin{equation*}
		\mu^n := \frac{1}{n} \sum_{i=1}^n \delta_{X_i^n}, \quad \mubar := \int_I \mu_u \, du.
	\end{equation*}
	If in addition Condition \ref{cond:limit-G}(b) holds, then
	\begin{equation}
		\label{eq:dense-X-cvg}
		\frac{1}{n} \sum_{i=1}^n \Emb \|X_i^n-X_{\frac{i}{n}}\|_{*,T}^2 \to 0
	\end{equation}
	as $n \to \infty$.
\end{Theorem}

\begin{Remark}
	\phantomsection
	\label{rmk:3.2}
	We note that the regularity assumption (Condition \ref{cond:limit-G}(b)) of $G$ is not needed to obtain the convergence in \eqref{eq:dense-mu-cvg}, but it helps obtain the $L^2$ convergence in \eqref{eq:dense-X-cvg}.
\end{Remark}

\begin{Remark}
	\phantomsection
	\label{rmk:3.3}
	We also note that the interaction $\{\xi^n_{ij}\}$ in Condition \ref{cond:dense-Gn}(b) could be either weighted or random.
	\begin{enumerate}[(a)]
	\item 
		When Condition \ref{cond:dense-Gn}(b.2) holds so that $\{\xi^n_{ij}\}$ are random, the convergence in Theorem \ref{thm:dense-LLN} (and Theorem \ref{thm:sparse-LLN}) is understood in the annealed sense, where the probability and expectation are taken with respect to all the randomness in the Brownian motions, initial states, and random $\{\xi^n_{ij}\}$.
	\item
		When Condition \ref{cond:dense-Gn}(b.1) holds so that $\{\xi^n_{ij}\}$ are deterministic, the convergence \eqref{eq:dense-mu-cvg} in Theorem \ref{thm:dense-LLN} (and \eqref{eq:sparse-mu-cvg} in Theorem \ref{thm:sparse-LLN}) could be interpreted in the quenched sense under suitable assumptions, given the realization of $\{\xi^n_{ij}\}$ possibly generated in a random manner.
		For example, suppose $\widetilde{G}_n$ is sampled from $G$ as follows: sample $n$ points $x_1^n,\dotsc,x_n^n$ i.i.d.\ from the uniform distribution on $I$ and relabel the indices so that $x_1^n \le x_2^n \le \dotsb \le x_n^n$.
		Define $\widetilde{G}_n$ as the step graphon with $\widetilde{G}_n(\frac{i}{n},\frac{j}{n})=$Bernoulli$(G(x_i^n,x_j^n))$ independently for $1 \le i \le j \le n$, and independent of $\{B_u,X_u(0):u \in I\}$.
		It then follows from \cite[Lemma 10.16]{Lovasz2012large} and a Borel--Cantelli lemma argument (see e.g.\ \cite[Remark 6]{Carmona2019stochastic} and \cite[Theorem 5 ]{PariseOzdaglar2019graphon}) that one has the following almost sure convergence in the cut distance
		$$\delta_\square(\widetilde{G}_n,G) \to 0, \quad \mbox{ as } n \to \infty.$$
		As illustrated in Remark \ref{rmk:3.1}, one can then find a relabeling $G_n$ of $\widetilde{G}_n$ such that $d_\square(G_n,G) \to 0$ almost surely, namely Condition \ref{cond:dense-Gn}(b.1) holds.
		Now suppose the initial states are i.i.d.\ so that the empirical measure is not affected by relabeling. Then one can apply \eqref{eq:dense-mu-cvg} to almost every graph realization $G_n$ to get a quenched law of large numbers, that is,
		$$\mu^n \to \mubar \text{ in } \Pmc(\Cmc_d) \text{ in } \Pmb[\cdot \mid \{G_m\}_{m \in \Nmb}] \text{ as } n \to \infty,$$
		for almost every graph realization $G_n$.
	\end{enumerate}
\end{Remark}

If the interaction $\xi_{ij}^n$ is sampled from a common graphon $G$ that is Lipschitz, one could get a uniform rate of convergence.

\begin{Condition}
	\label{cond:dense-Gn-special}
	{Suppose the strength of interaction $\xi_{ij}^n$ is generated as follows. 
	Either one of the following two cases holds for all $\{\xi_{ij}^n : i,j \in \{1,\dotsc,n\}, n \in \Nmb\}$:
	\begin{enumerate}[(i)]
	\item $\xi_{ij}^n=G(\frac{i}{n},\frac{j}{n})$;
	\item $\xi_{ij}^n=\xi_{ji}^n=$Bernoulli$(G(\frac{i}{n},\frac{j}{n}))$ independently for $1 \le i \le j \le n$, and independent of $\{B_u, X_u(0) : u \in I\}$.
	\end{enumerate}}
\end{Condition}

The proof of the following rate of convergence is given in Section \ref{sec:dense-LLN-special-pf}.
We note that the rate is consistent with that in the classic mean-field setup (in e.g.\ \cite{Sznitman1991}).

\begin{Theorem}
	\label{thm:dense-LLN-special}
	Suppose Conditions \ref{cond:standing}, \ref{cond:limit-G-special} and \ref{cond:dense-Gn-special} hold.
	Then there exists some $\kappa \in (0,\infty)$ such that
	\begin{equation}
		\label{eq:dense-special-X-cvg}
		\max_{i=1,\dotsc,n} \Emb \|X_i^n-X_{\frac{i}{n}}\|_{*,T}^2 \le \frac{\kappa}{n}, \quad \forall\, n \in \Nmb. 
	\end{equation} 
\end{Theorem}

\section{Mean-field particle systems on not-so-dense graphs}
\label{sec:sparse-system}

In this section we consider a sequence of $n$ interacting diffusions \eqref{eq:sparse-system-intro} with the strength of interaction governed by $\xi_{ij}^n$ associated with some kernel $G_n$ in a not-so-dense manner:
\begin{align}
	X_i^n(t) & = X_{\frac{i}{n}}(0) + \int_0^t \frac{1}{n\beta_n} \sum_{j=1}^n \xi_{ij}^n b(X_i^n(s),X_j^n(s))\,ds + \int_0^t \sigma(X_i^n(s))\,dB_{\frac{i}{n}}(s), \quad i \in \{1,\dotsc,n\}. 
	\label{eq:sparse-system}
\end{align}

We make the following assumptions on the initial states and system coefficients.
The forms and assumptions on $b$ and $\sigma$ are taken for technical requirements, in order to apply certain Girsanov's arguments.

\begin{Condition}
	\phantomsection
	\label{cond:standing-sparse}
	\begin{enumerate}[(a)]
	\item The map $I \ni u \mapsto \mu_u(0)=\Lmc(X_u(0)) \in \Pmc(\Rmb^d)$ is measurable. Moreover,
	\begin{equation*}
		\sup_{u \in I} \Emb |X_u(0)|^2 < \infty.
	\end{equation*}
	\item The function $b$ is bounded and Lipschitz.
	\item The function $\sigma$ is bounded, Lipschitz and invertible with bounded inverse $1/\sigma$.
	\item $\beta_n \in (0,1]$ and $\lim_{n \to \infty} n\beta_n = \infty$.
	\end{enumerate}
\end{Condition}

We make the following assumptions on the strength of interaction $\xi_{ij}^n$ and the associated kernel $G_n$.

\begin{Condition}
	\phantomsection
	\label{cond:sparse-Gn}
	\begin{enumerate}[(a)]
	\item 
		$G_n$ is a step graphon, that is, \eqref{eq:step-graphon} holds:
		$G_n(u,v) = G_n(\frac{\lceil nu \rceil}{n},\frac{\lceil nv \rceil}{n})$ for $(u,v) \in I \times I$.
	\item
		$\xi_{ij}^n=\xi_{ji}^n=$Bernoulli$(\beta_n G_n(\frac{i}{n},\frac{j}{n}))$ independently for $1 \le i \le j \le n$, and independent of $\{B_u, X_u(0) : u \in I\}$.
	\item
		$G_n \to G$ in the cut metric as $n \to \infty$.
	\end{enumerate}
\end{Condition}

\begin{Remark}
	\label{rmk:sparse}
	\begin{enumerate}[(a)]
	\item The interaction in \eqref{eq:sparse-system} is locally mean-field, where the strength of interaction between particles $i$ and $j$ is $\xi_{ij}^n$ scaled down by $n\beta_n$, the order of number of neighbors of $i$ or $j$. 
	\item Condition \ref{cond:sparse-Gn}(b) is also known as bond percolation, which has been studied for converging graph sequences in \cite{bollobas2010percolation}. 
	If we take $\lim_{n \to \infty} \beta_n = 0$ then graphs that we obtain are not-so-dense, in that for a graph $\mathbb{G}_n$ with an order of $n^2$ edges, the percolated graph will have approximately an order of $n^2\beta_n$ edges. 
	Therefore $\beta_n$ can be interpreted as the global sparsity parameter.
	\end{enumerate}
\end{Remark}

The limiting graphon particle system is given by
\begin{align}
	X_u(t) & = X_u(0) + \int_0^t \int_I \int_{\Rmb^d} b(X_u(s),x)G(u,v)\,\mu_{v,s}(dx)\,dv\,ds \notag \\
	& \quad + \int_0^t \sigma(X_u(s))\,dB_u(s), \quad \mu_{u,t}=\Lmc(X_u(t)), \quad u \in I.
	\label{eq:limit-system-sparse}
\end{align}
This is a special case of \eqref{eq:limit-system-dense} and hence Proposition \ref{prop:limit-existence-uniqueness} and Theorem \ref{thm:limit-system-property} (with Remark \ref{rmk:epsilon-1}) still hold.

The following theorem gives a LLN for the system \eqref{eq:sparse-system}.
The proof is given in Section \ref{sec:sparse-LLN-pf}.

\begin{Theorem}
	\label{thm:sparse-LLN}
	Suppose Conditions \ref{cond:limit-G}(a), \ref{cond:standing-sparse} and \ref{cond:sparse-Gn} hold.
	Then
	\begin{equation}
		\label{eq:sparse-mu-cvg}
		\mu^n \to \mubar \text{ in } \Pmc(\Cmc_d) \text{ in probability}
	\end{equation}
	as $n \to \infty$, where
	\begin{equation*}
		\mu^n := \frac{1}{n} \sum_{i=1}^n \delta_{X_i^n}, \quad \mubar := \int_I \mu_u \, du.
	\end{equation*}
	If in addition Condition \ref{cond:limit-G}(b) holds, then
	\begin{equation}
		\label{eq:sparse-X-cvg}
		\frac{1}{n} \sum_{i=1}^n \Emb \|X_i^n-X_{\frac{i}{n}}\|_{*,T}^2 \to 0
	\end{equation}
	as $n \to \infty$.
\end{Theorem}

If the interaction $\xi_{ij}^n$ is sampled from a common graphon $G$ that is Lipschitz, one could get a precise rate of convergence.

\begin{Condition}
	\label{cond:sparse-Gn-special}
	$\xi_{ij}^n=\xi_{ji}^n=$Bernoulli$(\beta_nG(\frac{i}{n},\frac{j}{n}))$ independently for $1 \le i \le j \le n$, and independent of $\{B_u, X_u(0) : u \in I\}$.
\end{Condition}

The proof of the following rate of convergence is given in Section \ref{sec:sparse-LLN-pf}.

\begin{Theorem}
	\label{thm:sparse-LLN-special}
	Suppose Conditions \ref{cond:limit-G-special}, \ref{cond:standing-sparse} and \ref{cond:sparse-Gn-special} hold.
		Then for each $q \in (1,\infty)$ there exists some $\kappa(q) \in (0,\infty)$ such that
		\begin{equation}
			\label{eq:sparse-special-X-cvg}
			\frac{1}{n}\sum_{i=1}^n \Emb \|X_i^n-X_{\frac{i}{n}}\|_{*,T}^2 \le \frac{\kappa(q)}{(n\beta_n)^{1/q}}, \quad \forall\, n \in \Nmb.
		\end{equation}
\end{Theorem}

\section{Proofs for Section \ref{sec:limit-system}}
\label{sec:limit-system-pf}

In this section we prove Proposition \ref{prop:limit-existence-uniqueness} and Theorem \ref{thm:limit-system-property}.

\subsection{Proof of Proposition \ref{prop:limit-existence-uniqueness}}
\label{sec:limit-existence-uniqueness-pf}

Define the map $\Mmc \ni \mu \mapsto \Phi(\mu) \in [\Pmc(\Cmc_d)]^I$ by $\Phi(\mu) := (\Lmc(X_u^\mu) : u \in I)$, where $X_u^\mu$ is the solution of
	\begin{align}
		X_u^\mu(t) & = X_u(0) + \int_0^t \int_I \int_{\Rmb^d} b(X_u^\mu(s),x)G(u,v)\,\mu_{v,s}(dx)\,dv\,ds \notag \\
		& \quad + \int_0^t \int_I \int_{\Rmb^d} \sigma(X_u^\mu(s),x)G(u,v)\,\mu_{v,s}(dx)\,dv\,dB_u(s). \label{eq:X-mu}
	\end{align}
	Note that the pathwise uniqueness of $\{X_u^\mu : u \in I\}$ is guaranteed by the Lipschitz properties of $b$ and $\sigma$ (see e.g.\ \cite[Theorem 5.2.5]{KaratzasShreve1991brownian}).
	We claim that 
	\begin{equation}
		\label{eq:well-pose-claim}
		\mbox{the pathwise existence of } \{X_u^\mu : u \in I\} \mbox{ holds and } \Phi(\mu) \in \Mmc \mbox{ for } \mu \in \Mmc.
	\end{equation}
	The proof of \eqref{eq:well-pose-claim} is deferred to the end.

	Next we claim that
	\begin{equation}
		\label{eq:graphon-Wass-claim}
		W_{2,t}^\Mmc(\Phi(\mu),\Phi(\nu)) \le \kappa \int_0^t W_{2,s}^\Mmc(\mu,\nu)\,ds, \quad \mu,\nu \in \Mmc.
	\end{equation}
	Note that the right hand side is well-defined since the integrand $W_{2,s}^\Mmc(\mu,\nu)$ is increasing and hence measurable in $s \in [0,T]$.
	To show \eqref{eq:graphon-Wass-claim}, consider the coupling
	\begin{align*}
		X_u^\mu(t) & = X_u(0) + \int_0^t \int_I \int_{\Rmb^d} b(X_u^\mu(s),x)G(u,v)\,\mu_{v,s}(dx)\,dv\,ds \\
		& \quad + \int_0^t \int_I \int_{\Rmb^d} \sigma(X_u^\mu(s),x)G(u,v)\,\mu_{v,s}(dx)\,dv\,dB_u(s), \\
		X_u^\nu(t) & = X_u(0) + \int_0^t \int_I \int_{\Rmb^d} b(X_u^\nu(s),x)G(u,v)\,\nu_{v,s}(dx)\,dv\,ds \\
		& \quad + \int_0^t \int_I \int_{\Rmb^d} \sigma(X_u^\nu(s),x)G(u,v)\,\nu_{v,s}(dx)\,dv\,dB_u(s).
	\end{align*}	
	It then follows from Holder's inequality and the Burkholder-Davis-Gundy inequality that
	\begin{align*}
		& \Emb \|X_u^\mu-X_u^\nu\|_{*,t}^2 \\
		& \le \kappa \Emb \int_0^t \int_I \left| \int_{\Rmb^d} b(X_u^\mu(s),x)G(u,v)\,\mu_{v,s}(dx) - \int_{\Rmb^d} b(X_u^\nu(s),x)G(u,v)\,\nu_{v,s}(dx) \right|^2 dv\,ds \\
		& \quad + \kappa \Emb \int_0^t \int_I \left| \int_{\Rmb^d} \sigma(X_u^\mu(s),x)G(u,v)\,\mu_{v,s}(dx) - \int_{\Rmb^d} \sigma(X_u^\nu(s),x)G(u,v)\,\nu_{v,s}(dx) \right|^2 dv\,ds.
	\end{align*}
	By adding and subtracting terms, we have
	\begin{align*}
		& \left| \int_{\Rmb^d} b(X_u^\mu(s),x)G(u,v)\,\mu_{v,s}(dx) - \int_{\Rmb^d} b(X_u^\nu(s),x)G(u,v)\,\nu_{v,s}(dx) \right|^2 \\
		& \le 2\left| \int_{\Rmb^d} [b(X_u^\mu(s),x)-b(X_u^\nu(s),x)]G(u,v)\,\mu_{v,s}(dx) \right|^2 \\
		& \quad + 2\left| \int_{\Rmb^d} b(X_u^\nu(s),x)G(u,v)\,[\mu_{v,s}-\nu_{v,s}](dx) \right|^2 \\
		& \le \kappa |X_u^\mu(s)-X_u^\nu(s)|^2 + \kappa [W_{2,s}^\Mmc(\mu,\nu)]^2,
	\end{align*}	
	where the last line uses the Lipschitz property of $b$ and Remark \ref{rmk:Wass-dual}.
	The same estimate holds when $b$ is replaced by $\sigma$ in the last display. 
	It then follows from Gronwall's inequality that
	\begin{equation*}
		\Emb \|X_u^\mu-X_u^\nu\|_{*,t}^2 \le \kappa \int_0^t [W_{2,s}^\Mmc(\mu,\nu)]^2\,ds.
	\end{equation*}
	Therefore the claim \eqref{eq:graphon-Wass-claim} holds.
	
	Using the claim \eqref{eq:graphon-Wass-claim}, we can immediately get pathwise uniqueness for the solution of \eqref{eq:limit-system-dense}.
	The pathwise existence also follows from \eqref{eq:graphon-Wass-claim} and a standard contraction argument (see e.g.\ \cite[Section I.1]{Sznitman1991}).
	To be precise, taking $\nu=(\Lmc(Y_u) : u \in I)$ where $Y_u(t) \equiv X_u(0)$ for $u \in I$ and $t \in [0,T]$, iterating \eqref{eq:graphon-Wass-claim} gives
	$$W_{2,T}^\Mmc(\Phi^{k+1}(\nu),\Phi^k(\nu)) \le \kappa^k \frac{T^k}{k!} W_{2,T}^\Mmc(\Phi(\nu),\nu), \quad k \in \Nmb.$$
	Using \eqref{eq:initial-moment}, the Lipschitz properties of $b$ and $\sigma$, and the fact that $\Phi(\nu)\in\Mmc$, one clearly has $W_{2,T}^\Mmc(\Phi(\nu),\nu) < \infty$.
	Therefore $\Phi^{k}(\nu)$ is a Cauchy sequence, and hence there exists some $\mu=(\mu_u)_{u \in I} \in [\Pmc(\Cmc_d)]^I$ such that $\lim_{k \to \infty} W_{2,T}^\Mmc(\Phi^k(\nu),\mu) = 0$ and $\sup_{u \in I} \int_{\Cmc_d} \|x\|_{*,T}^2 \,\mu_u(dx)<\infty$.
	This gives the existence in law of the solution of \eqref{eq:limit-system-dense}, which together with the pathwise uniqueness gives the pathwise existence of the solution of \eqref{eq:limit-system-dense}.
	{Using the claim \eqref{eq:well-pose-claim}, we have $\Phi^{k}(\nu) \in \Mmc$ and hence $u \mapsto [\Phi^{k}(\nu)]_u$ is measurable by the definition of $\Mmc$}.
	Since $\Pmc(\Cmc_d)$ is Polish and $\lim_{k \to \infty} W_{2,T}^\Mmc(\Phi^{k}(\nu),\mu)=0$, we have the measurability of $I \ni u \mapsto \mu_u \in \Pmc(\Cmc_d)$ (cf.\ \cite[Theorem 4.2.2]{Dudley2018real}).
	Using the Lipschitz properties of $b$ and $\sigma$, one has $\sup_{u \in I} \Emb\|X_u\|_{*,T}^{2+\varepsilon}<\infty$. 

	Finally we verify the claim \eqref{eq:well-pose-claim}.
	Fix $\mu \in \Mmc$.
	Let $\Xtil^0_u(t) \equiv X_u(0)$ for $t \in [0,T]$ and $u \in I$.
	For $n \in \Nmb$, let
	\begin{align}
		\Xtil^n_u(t) & = \Xtil^{n-1}_u(0) + \int_0^t \int_I \int_{\Rmb^d} b(\Xtil^{n-1}_u(s),x)G(u,v)\,\mu_{v,s}(dx)\,dv\,ds \notag \\
		& \quad + \int_0^t \int_I \int_{\Rmb^d} \sigma(\Xtil^{n-1}_u(s),x)G(u,v)\,\mu_{v,s}(dx)\,dv\,dB_u(s). \label{eq:Xtiln}
	\end{align}
	Since $\mu \in \Mmc$, it follows that $\Xtil_u^n$ and integrals in \eqref{eq:Xtiln} are well-defined for all $u \in I$ and $n \in \Nmb$.
	We will prove by induction that 
	\begin{equation}
		\label{eq:claim-measurable}
		I \ni u \mapsto \Lmc(\Xtil^n_u, B_u) \in \Pmc(\Cmc_d \times \Cmc_d) \mbox{ is measurable for each } n =0,1,\dotsc
	\end{equation}
	By construction and Condition \ref{cond:standing}(a), \eqref{eq:claim-measurable} holds for $n=0$.
	Next suppose \eqref{eq:claim-measurable} holds up to $k-1$ for some $k \in \Nmb$.
	To complete the proof of \eqref{eq:claim-measurable}, it suffices 
	to show that
	$$I \ni u \mapsto \Lmc(\Xtil^k_u(t_1),B_u(t_1),\dotsc,\Xtil^k_u(t_m),B_u(t_m)) \in \Pmc((\Rmb^d \times \Rmb^d)^m)$$ 
	is measurable for all $0 \le t_1 \le \dotsb \le t_m \le T$ and $m \in \Nmb$. 
	It then suffices to show that
	$$I \ni u \mapsto \Emb\left[\prod_{i=1}^m \left( f_i(\Xtil^k_u(t_i))g_i(B_u(t_i)) \right)\right] \in \Rmb$$
	is measurable for all $0 \le t_1 \le \dotsb \le t_m \le T$, $m \in \Nmb$ and bounded and continuous functions $\{f_i,g_i : i=1,\dotsc,m\}$ on $\Rmb^d$. 
	Now consider the following auxiliary processes
	\begin{align*}
		\Xtil^{k,\delta}_u(t) & = \Xtil^{k-1}_u(0) + \int_0^t \int_I \int_{\Rmb^d} b(\Xtil^{k-1}_u(\lfl \frac{s}{\delta} \rfl \delta),x)G(u,v)\,\mu_{v,\lfl \frac{s}{\delta} \rfl \delta}(dx)\,dv\,ds \\
		& \quad + \int_0^t \int_I \int_{\Rmb^d} \sigma(\Xtil^{k-1}_u(\lfl \frac{s}{\delta} \rfl \delta),x)G(u,v)\,\mu_{v,\lfl \frac{s}{\delta} \rfl \delta}(dx)\,dv\,dB_u(s),
	\end{align*}		
	where $\delta \in (0,1)$.
	Clearly, $\Xtil^{k,\delta}_u(t)$ converges to $\Xtil^{k}_u(t)$ in probability as $\delta \to 0$ for each $u \in I$.
	So it suffices to prove that
	$$I \ni u \mapsto \Emb\left[\prod_{i=1}^m \left( f_i(\Xtil^{k,\delta}_u(t_i))g_i(B_u(t_i)) \right)\right] \in \Rmb$$
	is measurable for all $0 \le t_1 \le \dotsb \le t_m \le T$, $m \in \Nmb$ and bounded and continuous functions $\{f_i,g_i : i=1,\dotsc,m\}$ on $\Rmb^d$.	
	Fix $t \in [0,T]$.
	Since \eqref{eq:claim-measurable} holds for $k-1$, it further suffices to show that
	$$\Xtil^{k,\delta}_u(t) = h(u,\Xtil^{k-1}_u,B_u)$$
	for some measurable function $h \colon I \times \Cmc_d \times \Cmc_d \to \Rmb$.	
	{Noting that $\Xtil^{k,\delta}_u$ is in fact a finite sum of terms depending on $\{\Xtil^{k-1}_u(0), \Xtil^{k-1}_u(\delta), \dotsc, \Xtil^{k-1}_u(\lfl \frac{t}{\delta} \rfl \delta)\}$ and $\{B_u(0), B_u(\delta), \dotsc, B_u(\lfl \frac{t}{\delta} \rfl \delta)\}$ continuously}, we have that $h(u,\cdot,\cdot)$ is continuous on $\Cmc_d \times \Cmc_d$ for each $u \in I$, and that $h(\cdot,x,w)$ is measurable on $I$ for each $(x,w) \in \Cmc_d \times \Cmc_d$.
	Therefore $h$ is measurable and this verifies \eqref{eq:claim-measurable} by induction.
	Using the Lipschitz Properties of $b$ and $\sigma$, from \eqref{eq:Xtiln} we can get
	\begin{equation*}
		\sup_{u \in I} \Emb \|\Xtil^{n+1}_u-\Xtil^n_u\|_{*,t}^2 \le \kappa \int_0^t \sup_{u \in I} \Emb \|\Xtil^n_u-\Xtil^{n-1}_u\|_{*,s}^2 \,ds. 
	\end{equation*}
	Therefore $\{\Xtil^n_u : n \in \Nmb\}$ is Cauchy and converges uniformly in $u \in I$ in probability to some $X_u^\mu$ that satisfies \eqref{eq:X-mu}.
	Since $\Pmc(\Cmc_d)$ is Polish, the measurability of $u \mapsto \Lmc(\Xtil^n_u)$ then guarantees that $I \ni u \mapsto \Lmc(X_u^\mu) \in \Pmc(\Cmc_d)$ is measurable (cf.\ \cite[Theorem 4.2.2]{Dudley2018real}).
	Also noting that $\sup_{n \in \Nmb} \sup_{u \in I} \Emb \|\Xtil^n_u\|_{*,T}^2 < \infty$, we see that	
	$\Phi$ is actually a well-defined map from $\Mmc$ to $\Mmc$.
	This verifies \eqref{eq:well-pose-claim} and completes the proof of Proposition \ref{prop:limit-existence-uniqueness}.
	\qed

From the proof of the claim \eqref{eq:well-pose-claim} we can also get the following joint measurability, which will be used in the proof of Theorem \ref{thm:limit-system-property}. 

\begin{Remark}
	\label{rmk:joint-measurable}
	Given graphons $G_i$ and measures $\mu^{G_i}$, $i=1,2$, let $\Xtil^{G_i,0}_u(t) \equiv X_u(0)$ for $t \in [0,T]$ and $u \in I$.
	For $n \in \Nmb$, let
	\begin{align*}
		\Xtil^{G_i,n}_u(t) & = \Xtil^{G_i,n-1}_u(0) + \int_0^t \int_I \int_{\Rmb^d} b(\Xtil^{G_i,n-1}_u(s),x)G_i(u,v)\,\mu^{G_i}_{v,s}(dx)\,dv\,ds \\
		& \quad + \int_0^t \int_I \int_{\Rmb^d} \sigma(\Xtil^{G_i,n-1}_u(s),x)G_i(u,v)\,\mu^{G_i}_{v,s}(dx)\,dv\,dB_u(s).
	\end{align*}
	A similar proof to that of \eqref{eq:claim-measurable} shows that $I \ni u \mapsto \Lmc(\Xtil^{G_1,n}_u, \Xtil^{G_2,n}_u, B_u) \in \Pmc(\Cmc_d \times \Cmc_d \times \Cmc_d) \mbox{ is measurable for each } n =0,1,\dotsc$.
	It then follows that $\{(\Xtil^{G_1,n}_u, \Xtil^{G_2,n}_u): n \in \Nmb\}$ is Cauchy and converges uniformly in $u \in I$ in probability to some $(X_u^{G_1,\mu^{G_1}}, X_u^{G_2,\mu^{G_2}})$, where $X_u^{G_i,\mu^{G_i}}$ satisfies \eqref{eq:X-mu} with $G$ and $\mu$ replaced by $G_i$ and $\mu^{G_i}$.
	Therefore we have the joint measurability of $I \ni u \mapsto \Lmc(X_u^{G_1,\mu^{G_1}}, X_u^{G_2,\mu^{G_2}}) \in \Pmc(\Cmc_d \times \Cmc_d)$.
\end{Remark}

\subsection{Proof of Theorem \ref{thm:limit-system-property}}
\label{sec:limit-system-property-pf}

	(a) (b) Fix $u_1,u_2 \in I$. 
	Consider the following diffusions:
	\begin{align*}
		\Xtil_{u_1}(t) & = \Xtil_{u_1}(0) + \int_0^t \int_I \int_{\Rmb^d} b(\Xtil_{u_1}(s),x)G(u_1,v)\,\mu_{v,s}(dx)\,dv\,ds \\
		& \quad + \int_0^t \int_I \int_{\Rmb^d} \sigma(\Xtil_{u_1}(s),x)G(u_1,v)\,\mu_{v,s}(dx)\,dv\,dB(s), \\
		\Xtil_{u_2}(t) & = \Xtil_{u_2}(0) + \int_0^t \int_I \int_{\Rmb^d} b(\Xtil_{u_2}(s),x)G(u_2,v)\,\mu_{v,s}(dx)\,dv\,ds \\
		& \quad + \int_0^t \int_I \int_{\Rmb^d} \sigma(\Xtil_{u_2}(s),x)G(u_2,v)\,\mu_{v,s}(dx)\,dv\,dB(s).		
	\end{align*}
	Here $B$ is a $d$-dimensional Brownian motion independent of $\{\Xtil_{u_1}(0), \Xtil_{u_2}(0)\}$, $\Lmc(\Xtil_{u_1}(0)) = \mu_{u_1}(0)$, $\Lmc(\Xtil_{u_2}(0)) = \mu_{u_2}(0)$, but $\Xtil_{u_1}(0)$ and $\Xtil_{u_2}(0)$ may not be independent. 
	From Proposition \ref{prop:limit-existence-uniqueness} we have $\Lmc(\Xtil_{u_1}) = \mu_{u_1}$ and $\Lmc(\Xtil_{u_2}) = \mu_{u_2}$.
	Also note that
	\begin{align*}
		& \Emb \|\Xtil_{u_1}-\Xtil_{u_2}\|_{*,t}^2 \\
		& \le \kappa \Emb |\Xtil_{u_1}(0)-\Xtil_{u_2}(0)|^2 \\
		& \quad + \kappa\Emb \int_0^t \int_I \int_{\Rmb^d}  \left| b(\Xtil_{u_1}(s),x)G(u_1,v) - b(\Xtil_{u_2}(s),x)G(u_2,v) \right|^2 \mu_{v,s}(dx)\,dv\,ds \\
		& \quad + \kappa \Emb \int_0^t \int_I \int_{\Rmb^d}  \left| \sigma(\Xtil_{u_1}(s),x)G(u_1,v) - \sigma(\Xtil_{u_2}(s),x)G(u_2,v) \right|^2 \mu_{v,s}(dx)\,dv\,ds \\
		& \le \kappa \Emb |\Xtil_{u_1}(0)-\Xtil_{u_2}(0)|^2 + \kappa \Emb \int_0^t |\Xtil_{u_1}(s)-\Xtil_{u_2}(s)|^2\,ds + \kappa \int_I |G(u_1,v) - G(u_2,v)|^2\,dv,
	\end{align*}
	where the last line follows on adding and subtracting terms, using the Lipschitz properties of $b$ and $\sigma$, and uniformly finite second moment of $\mu_{u}$.
	It then follows from Gronwall's inequality that 
	\begin{equation*}
		[W_{2,T}(\mu_{u_1}, \mu_{u_2})]^2 \le \Emb \|\Xtil_{u_1}-\Xtil_{u_2}\|_{*,T}^2 \le \kappa \Emb |\Xtil_{u_1}(0)-\Xtil_{u_2}(0)|^2 + \kappa \int_I |G(u_1,v) - G(u_2,v)|^2\,dv.
	\end{equation*}	
	Taking the infimum over random variables $\Xtil_{u_1}(0)$ and $\Xtil_{u_2}(0)$ such that $\Lmc(\Xtil_{u_1}(0)) = \mu_{u_1}(0)$ and $\Lmc(\Xtil_{u_2}(0)) = \mu_{u_2}(0)$, we have
	\begin{equation*}
		[W_{2,T}(\mu_{u_1}, \mu_{u_2})]^2 \le \kappa [W_2(\mu_{u_1}(0),\mu_{u_2}(0))]^2 + \kappa \int_I |G(u_1,v) - G(u_2,v)|^2\,dv.
	\end{equation*}
	Part (a) and Part (b) then follow from Condition \ref{cond:limit-G} and Condition \ref{cond:limit-G-special}, respectively.
	
	(c) Fix $G_n \to G$ in the cut metric as $n \to \infty$.
	Let $X^{G_n}, \mu^{G_n}$ (resp.\ $X^G,\mu^G$) be the solution of \eqref{eq:limit-system-dense} associated with the graphon $G_n$ (resp.\ $G$).	
	Fix $t \in [0,T]$.
	Using Remark \ref{rmk:joint-measurable}, the left hand side below is well-defined and we have
	\begin{align}
		\int_I \Emb \|X^{G_n}_u-X^G_u\|_{*,t}^2\,du 
		& \le \kappa \int_0^t \int_I \Emb \left| \int_I \int_{\Rmb^d} b(X_u^{G_n}(s),x)G_n(u,v)\,\mu^{G_n}_{v,s}(dx)\,dv \right. \notag\\
		& \qquad \left. - \int_I \int_{\Rmb^d} b(X_u^G(s),x)G(u,v)\,\mu^G_{v,s}(dx)\,dv \right|^2 du\,ds \notag \\
		& \quad + \kappa \int_0^t \int_I \Emb \left| \int_I \int_{\Rmb^d} \sigma(X_u^{G_n}(s),x)G_n(u,v)\,\mu^{G_n}_{v,s}(dx)\,dv \right. \notag \\
		& \qquad \left. - \int_I \int_{\Rmb^d} \sigma(X_u^G(s),x)G(u,v)\,\mu^G_{v,s}(dx)\,dv \right|^2 du\,ds. \label{eq:graphon1}
	\end{align}
	We will analyze the first integrand above for fixed $s \in [0,t]$, and the analysis for $\sigma$ is similar.
	By adding and subtracting terms, we have
	\begin{align}
		& \int_I \Emb \left| \int_I \int_{\Rmb^d} b(X_u^{G_n}(s),x)G_n(u,v)\,\mu^{G_n}_{v,s}(dx)\,dv - \int_I \int_{\Rmb^d} b(X_u^G(s),x)G(u,v)\,\mu^G_{v,s}(dx)\,dv \right|^2 du \notag \\
		& \le \kappa \int_I \int_I \Emb \left| \int_{\Rmb^d} [b(X_u^{G_n}(s),x)-b(X_u^G(s),x)]G_n(u,v)\,\mu^{G_n}_{v,s}(dx) \right|^2 dv\,du \notag \\
		& \quad + \kappa \int_I \int_I \Emb \left| \int_{\Rmb^d} b(X_u^G(s),x)G_n(u,v)\,[\mu^{G_n}_{v,s}-\mu^G_{v,s}](dx) \right|^2 dv\,du \notag \\
		& \quad + \kappa \int_I \Emb \left| \int_I \int_{\Rmb^d} b(X_u^G(s),x)[G_n(u,v)-G(u,v)]\,\mu^G_{v,s}(dx)\,dv \right|^2 du \notag \\
		& =: \kappa \left( \Jmc_s^{n,1} + \Jmc_s^{n,2} + \Jmc_s^{n,3} \right). \label{eq:graphon2}
	\end{align}
	
	Now we analyze each term $\Jmc^{n,k}_s$, $k=1,2,3$.
	Using the Lipschitz property of $b$ we have
	\begin{equation}
		\Jmc_s^{n,1} 
		\le \kappa \int_I \Emb |X_u^{G_n}(s) - X_u^G(s)|^2 \, du. \label{eq:graphon3}
	\end{equation}
	Using Remark \ref{rmk:Wass-dual} and the Lipschitz property of $b$ we have
	\begin{equation}
		\Jmc_s^{n,2} 
		\le \kappa \int_I [W_{2,s}(\mu_v^{G_n},\mu_v^G)]^2 \, dv. \label{eq:graphon4}
	\end{equation}	
	For the last term $\Jmc^{n,3}_s$, 
	fix $M \in (0,\infty)$ and write 
	\begin{equation}
		\label{eq:bM}
		b_M(x,y) := b(x,y) \one_{\{|x|\le M, |y|\le M\}}.
	\end{equation}
	Since $b$ is Lipschitz continuous, it follows from \cite[Corollary 2 of Theorem 3.1]{Schultz1969multivariate} that there exist some $m=m(M) \in \Nmb$ and polynomials
	\begin{equation}
		\label{eq:bm}
		\btil_m(x,y) := \sum_{k=1}^m a_k(x) c_k(y) \one_{\{|x|\le M, |y|\le M\}},
	\end{equation} 
	where $a_k(x)$ and $c_k(y)$ are powers of $x$ and $y$ {respectively}, for each $k=1,\dotsc,m$,
	such that
	\begin{equation}
		\label{eq:graphon-bm}
		|b_M(x,y)-\btil_m(x,y)| \le 1/M, \quad {x,y \in \Rmb^d}.
	\end{equation}
	By adding and subtracting terms, we have
	\begin{align}
		\Jmc_s^{n,3} 
		& \le \kappa \int_I \int_I \Emb \left| \int_{\Rmb^d} [b(X_u^G(s),x)-b_M(X_u^G(s),x)] [G_n(u,v)-G(u,v)]\,\mu^G_{v,s}(dx) \right|^2 dv\,du \notag \\
		& + \kappa \int_I \int_I \Emb \left| \int_{\Rmb^d} [b_M(X_u^G(s),x)-\btil_m(X_u^G(s),x)] [G_n(u,v)-G(u,v)]\,\mu^G_{v,s}(dx) \right|^2 dv\,du \notag \\
		& + \kappa \int_I \Emb \left| \int_I \int_{\Rmb^d} \btil_m(X_u^G(s),x) [G_n(u,v)-G(u,v)]\,\mu^G_{v,s}(dx)\,dv \right|^2 du \notag \\
		& =: \kappa \sum_{k=1}^3 \Jmc_s^{n,3,k}. \label{eq:graphon5}
	\end{align}
	Next we analyze each term $\Jmc_s^{n,3,k}$, $k=1,2,3$.
	For $\Jmc_s^{n,3,1}$, using the Lipschitz property of $b$,
	\eqref{eq:bM} and Proposition \ref{prop:limit-existence-uniqueness} we have
	\begin{equation}
		\Jmc_s^{n,3,1} \le \kappa \int_I \int_I \Emb \left| \int_{\Rmb^d} (1+|X_u^G(s)|+|x|)(\one_{\{|X_u^G(s)|>M\}}+\one_{\{|x|>M\}})\,\mu_{v,s}^G(dx)\right|^2dv\,du \le \frac{\kappa}{M^\varepsilon}. \label{eq:graphon6}
	\end{equation}
	For $\Jmc_s^{n,3,2}$, using \eqref{eq:graphon-bm} we have
	\begin{equation}
		\Jmc_s^{n,3,2} \le \frac{\kappa}{M^2}. \label{eq:graphon7}
	\end{equation}
	For $\Jmc_s^{n,3,3}$, using the definition of the bounded function $\btil_m$ in \eqref{eq:bm} we have
	\begin{align}
		\Jmc_s^{n,3,3} & \le \kappa(M) \int_I \Emb \left| \int_I \int_{\Rmb^d} \btil_m(X_u^G(s),x) [G_n(u,v)-G(u,v)]\,\mu^G_{v,s}(dx)\,dv \right| du \notag \\ 
		& \le \kappa(M) \sum_{k=1}^m \int_I \Emb \left[|a_k(X_u^G(s))|\one_{\{|X_u^G(s)|\le M\}}\right] \notag \\
		& \qquad \cdot \left| \int_I \left[ G_n(u,v) - G(u,v) \right] \left[ \int_{\Rmb^d} c_k(x)\one_{\{|x|\le M\}} \mu_{v,s}^G(dx) \right] dv \right| du \notag \\
		& \le \kappa(M) \|G_n-G\|, \label{eq:graphon8}
	\end{align}
	where $\kappa(M)$ is some constant that depends on $M$ but not on $n$.
	Combining \eqref{eq:graphon1}--\eqref{eq:graphon4} and \eqref{eq:graphon5}--\eqref{eq:graphon8} with Gronwall's inequality, we have
	\begin{align*}
		\int_I [W_{2,t}(\mu_u^{G_n},\mu_u^G)]^2 \,du & \le \int_I \Emb \|X^{G_n}_u-X^G_u\|_{*,t}^2\,du \\
		& \le \kappa \left( \int_0^t \int_I [W_{2,s}(\mu_u^{G_n},\mu_u^G)]^2 \,du\,ds + \frac{1}{M^\varepsilon} + \kappa(M) \|G_n-G\| \right).
	\end{align*}
	It then follows from the Gronwall's inequality again that
	\begin{equation*}
		\int_I [W_{2,t}(\mu_u^{G_n},\mu_u^G)]^2 \,du \le \kappa \left( \frac{1}{M^\varepsilon} + \kappa(M) \|G_n-G\| \right).
	\end{equation*}
	Since $G_n \to G$ in the cut metric, from Remark \ref{rmk:graphon-convergence} we have $\|G_n-G\| \to 0$ as $n \to \infty$.
	Therefore, by taking $\limsup_{n \to \infty}$ and then $\limsup_{M \to \infty}$ in the last display, we have the desired result.
	This gives part (c) and completes the proof of Theorem \ref{thm:limit-system-property}.
	\qed

\begin{Remark}
	\label{rmk:epsilon-1}
	The assumption that $\varepsilon>0$ is used in \eqref{eq:graphon6}. 
	If $b$ and $\sigma$ are bounded functions, then one can allow $\varepsilon=0$ and replace the estimate in \eqref{eq:graphon6} by 
	$$\Jmc_s^{n,3,1} \le \kappa \int_I \int_I \Emb \left| \int_{\Rmb^d} (\one_{\{|X_u^G(s)|>M\}}+\one_{\{|x|>M\}})\,\mu_{v,s}^G(dx)\right|^2dv\,du \le \frac{\kappa}{M^2}.$$
\end{Remark}

\section{Proofs for Section \ref{sec:dense-system}}
\label{sec:dense-system-pf}

In this section we prove Theorems \ref{thm:dense-LLN} and \ref{thm:dense-LLN-special}.

\subsection{Preliminary Estimates}

We first provide some preliminary estimates under Condition \ref{cond:limit-G}.

\begin{Lemma}
	\label{lem:new1}
	Suppose Conditions \ref{cond:standing}, \ref{cond:limit-G} and \ref{cond:dense-Gn}(a,b) hold.
	Then there exist some $\kappa, \kappa(M) \in (0,\infty)$ for each $M \in (0,\infty)$ such that
	\begin{equation*}
		\limsup_{n \to \infty} \frac{1}{n} \sum_{i=1}^n \Emb \|X_i^n-X_{\frac{i}{n}}\|_{*,T}^2 \le \kappa(M) \limsup_{n \to \infty} \|G_n-G\| + \frac{\kappa}{M^\varepsilon}.
	\end{equation*}
\end{Lemma}

\begin{proof}
	Fix $t \in [0,T]$.
	\begin{align}
		& \frac{1}{n} \sum_{i=1}^n \Emb \|X_i^n-X_{\frac{i}{n}}\|_{*,t}^2 \notag \\
		& \le \kappa \int_0^t \left[ \frac{1}{n} \sum_{i=1}^n \Emb \left| \frac{1}{n} \sum_{j=1}^n \xi_{ij}^n b(X_i^n(s),X_j^n(s)) - \int_I \int_{\Rmb^d} b(X_{\frac{i}{n}}(s),x)G(\frac{i}{n},v)\,\mu_{v,s}(dx)\,dv \right|^2 \right] ds \notag \\
		& \quad + \kappa \int_0^t \left[ \frac{1}{n} \sum_{i=1}^n \Emb \left| \frac{1}{n} \sum_{j=1}^n \xi_{ij}^n \sigma(X_i^n(s),X_j^n(s)) - \int_I \int_{\Rmb^d} \sigma(X_{\frac{i}{n}}(s),x)G(\frac{i}{n},v)\,\mu_{v,s}(dx)\,dv \right|^2 \right] ds. \label{eq:dense1}
	\end{align}
	We will analyze the first integrand above for fixed $s \in [0,t]$, and the analysis for $\sigma$ is similar.
	By adding and subtracting terms, we have
	\begin{align}
		& \frac{1}{n} \sum_{i=1}^n \Emb \left| \frac{1}{n} \sum_{j=1}^n \xi_{ij}^n b(X_i^n(s),X_j^n(s)) - \int_I \int_{\Rmb^d} b(X_{\frac{i}{n}}(s),x)G(\frac{i}{n},v)\,\mu_{v,s}(dx)\,dv \right|^2 \notag \\
		& \le \frac{3}{n} \sum_{i=1}^n \Emb \left| \frac{1}{n} \sum_{j=1}^n \xi_{ij}^n \left( b(X_i^n(s),X_j^n(s)) - b(X_{\frac{i}{n}}(s),X_{\frac{j}{n}}(s)) \right) \right|^2 \notag \\
		& \quad + \frac{3}{n} \sum_{i=1}^n \Emb \left| \frac{1}{n} \sum_{j=1}^n \left( \xi_{ij}^n b(X_{\frac{i}{n}}(s),X_{\frac{j}{n}}(s)) - \int_{\Rmb^d} b(X_{\frac{i}{n}}(s),x)G_n(\frac{i}{n},\frac{j}{n})\,\mu_{\frac{j}{n},s}(dx)\right) \right|^2 \notag \\
		& \quad + \frac{3}{n} \sum_{i=1}^n \Emb \left| \frac{1}{n} \sum_{j=1}^n \int_{\Rmb^d} b(X_{\frac{i}{n}}(s),x)G_n(\frac{i}{n},\frac{j}{n})\,\mu_{\frac{j}{n},s}(dx) - \int_I \int_{\Rmb^d} b(X_{\frac{i}{n}}(s),x)G(\frac{i}{n},v)\,\mu_{v,s}(dx)\,dv \right|^2 \notag \\
		& =: 3 \left( \Tmc_s^{n,1} + \Tmc_s^{n,2} + \Tmc_s^{n,3} \right). \label{eq:dense2}
	\end{align}

	Fix $s \in [0,T]$.
	For $\Tmc_s^{n,1}$, using the Lipschitz property of $b$ we have
	\begin{align}
		\Tmc_s^{n,1} 
		& \le \kappa \frac{1}{n} \sum_{i=1}^n \Emb \left[ \frac{1}{n} \sum_{j=1}^n \left( |X_i^n(s)-X_{\frac{i}{n}}(s)|^2 + |X_j^n(s) - X_{\frac{j}{n}}(s)|^2 \right) \right] \notag \\
		& \le \kappa \frac{1}{n} \sum_{i=1}^n  \Emb |X_i^n(s) - X_{\frac{i}{n}}(s)|^2. \label{eq:dense-T1}
	\end{align}
	
	For $\Tmc_s^{n,2}$, using a weak LLN type argument, we have
	\begin{align}
		\Tmc_s^{n,2} 
		& 
		= \frac{1}{n} \sum_{i=1}^n \frac{1}{n^2} \sum_{j=1}^n \sum_{k=1}^n \Emb \left[ \left( \xi_{ij}^n b(X_{\frac{i}{n}}(s),X_{\frac{j}{n}}(s)) - \int_{\Rmb^d} b(X_{\frac{i}{n}}(s),x)G_n(\frac{i}{n},\frac{j}{n})\,\mu_{\frac{j}{n},s}(dx)\right) \right. \notag \\
		& \qquad \left. \cdot \left( \xi_{ik}^n b(X_{\frac{i}{n}}(s),X_{\frac{k}{n}}(s)) - \int_{\Rmb^d} b(X_{\frac{i}{n}}(s),x)G_n(\frac{i}{n},\frac{k}{n})\,\mu_{\frac{k}{n},s}(dx)\right) \right] \notag \\
		& = \frac{1}{n} \sum_{i=1}^n \frac{1}{n^2} \sum_{j=1}^n \sum_{k \in \{i,j\}} \Emb \left[ \left( \xi_{ij}^n b(X_{\frac{i}{n}}(s),X_{\frac{j}{n}}(s)) - \int_{\Rmb^d} b(X_{\frac{i}{n}}(s),x)G_n(\frac{i}{n},\frac{j}{n})\,\mu_{\frac{j}{n},s}(dx)\right) \right. \notag \\
		& \qquad \left. \cdot \left( \xi_{ik}^n b(X_{\frac{i}{n}}(s),X_{\frac{k}{n}}(s)) - \int_{\Rmb^d} b(X_{\frac{i}{n}}(s),x)G_n(\frac{i}{n},\frac{k}{n})\,\mu_{\frac{k}{n},s}(dx)\right) \right] \notag \\
		& \le \frac{\kappa}{n}, \label{eq:dense-T2}
	\end{align}
	where the second equality follows from the observation that the expectation is zero whenever $k \notin \{i,j\}$ by Condition \ref{cond:dense-Gn}(b) and the independence of $\{X_{\frac{i}{n}}\}$ and $\{\xi_{ij}^n\}$, and the inequality uses the boundedness of $\xi_{ij}^n, G_n$, Lipschitz property of $b$, and the uniformly finite second moment of $X_u$.

	The analysis of $\Tmc_s^{n,3}$ is similar to that of $\Jmc^{n,3}_s$ in the proof of Theorem \ref{thm:limit-system-property}(c) but is more involved.
	Fix $M \in (0,\infty)$ and let $b_M$ and $\btil_m$ be defined as in \eqref{eq:bM} and \eqref{eq:bm}, such that \eqref{eq:graphon-bm} holds.
	By adding and subtracting terms, we have
	\begin{align}
		\Tmc_s^{n,3} 
		& \le \frac{\kappa}{n} \sum_{i=1}^n \Emb \left| \frac{1}{n} \sum_{j=1}^n \int_{\Rmb^d} [b(X_{\frac{i}{n}}(s),x)-b_M(X_{\frac{i}{n}}(s),x)]G_n(\frac{i}{n},\frac{j}{n})\,\mu_{\frac{j}{n},s}(dx) \right|^2 \notag \\
		& \quad + \frac{\kappa}{n} \sum_{i=1}^n \Emb \left| \int_I \int_{\Rmb^d} [b(X_{\frac{i}{n}}(s),x)-b_M(X_{\frac{i}{n}}(s),x)]G(\frac{i}{n},v)\,\mu_{v,s}(dx)\,dv \right|^2 \notag \\
		& \quad + \frac{\kappa}{n} \sum_{i=1}^n \Emb \left| \frac{1}{n} \sum_{j=1}^n \int_{\Rmb^d} [b_M(X_{\frac{i}{n}}(s),x)-\btil_m(X_{\frac{i}{n}}(s),x)]G_n(\frac{i}{n},\frac{j}{n})\,\mu_{\frac{j}{n},s}(dx) \right|^2 \notag \\
		& \quad + \frac{\kappa}{n} \sum_{i=1}^n \Emb \left| \int_I \int_{\Rmb^d} [b_M(X_{\frac{i}{n}}(s),x)-\btil_m(X_{\frac{i}{n}}(s),x)]G(\frac{i}{n},v)\,\mu_{v,s}(dx)\,dv \right|^2 \notag \\
		& \quad + \frac{\kappa}{n} \sum_{i=1}^n \Emb \left| \frac{1}{n} \sum_{j=1}^n \int_{\Rmb^d} \btil_m(X_{\frac{i}{n}}(s),x)G_n(\frac{i}{n},\frac{j}{n})\,\mu_{\frac{j}{n},s}(dx) \right. \notag \\
		& \left. \qquad - \int_I \int_{\Rmb^d} \btil_m(X_{\frac{i}{n}}(s),x)G(\frac{i}{n},v)\,\mu_{v,s}(dx)\,dv \right|^2 \notag \\
		& =: \kappa \sum_{k=1}^5 \Tmc_s^{n,3,k}. \label{eq:dense-T30}
	\end{align}
	Next we analyze each term.
	For $\Tmc_s^{n,3,1}$ and $\Tmc_s^{n,3,2}$, using Proposition \ref{prop:limit-existence-uniqueness}
	and \eqref{eq:bM} we have
	\begin{align}
		\Tmc_s^{n,3,1} & \le  \frac{\kappa}{n} \sum_{i=1}^n \Emb \left| \frac{1}{n} \sum_{j=1}^n \int_{\Rmb^d} (1+|X_{\frac{i}{n}}(s)|+|x|)[\one_{\{|X_{\frac{i}{n}}(s)|>M\}}+\one_{\{|x|>M\}}]\,\mu_{\frac{j}{n},s}(dx) \right|^2 \le \frac{\kappa}{M^\varepsilon}, \label{eq:dense-T31} \\
		\Tmc_s^{n,3,2} & \le  \frac{\kappa}{n} \sum_{i=1}^n \Emb \left| \int_I \int_{\Rmb^d} (1+|X_{\frac{i}{n}}(s)|+|x|)[\one_{\{|X_{\frac{i}{n}}(s)|>M\}}+\one_{\{|x|>M\}}]\,\mu_{v,s}(dx)\,dv \right|^2 \le \frac{\kappa}{M^\varepsilon}. \label{eq:dense-T32}
	\end{align}
	For $\Tmc_s^{n,3,3}$ and $\Tmc_s^{n,3,4}$, using \eqref{eq:graphon-bm} we have
	\begin{equation}
		\Tmc_s^{n,3,3} \le \frac{\kappa}{M^2}, \quad \Tmc_s^{n,3,4} \le \frac{\kappa}{M^2} \label{eq:dense-T33-T34}.
	\end{equation}
	For $\Tmc_s^{n,3,5}$, using the step graphon structure \eqref{eq:step-graphon} of $G_n$ we have
	\begin{align}
		\Tmc_s^{n,3,5} & = \int_I \Emb \left| \int_I \int_{\Rmb^d} \btil_m(X_{\frac{\lceil nu \rceil}{n}}(s),x)G_n(u,v)\,\mu_{\frac{\lceil nv \rceil}{n},s}(dx)\,dv \right. \notag \\
		& \qquad \left. - \int_I \int_{\Rmb^d} \btil_m(X_{\frac{\lceil nu \rceil}{n}}(s),x)G(\frac{\lceil nu \rceil}{n},v)\,\mu_{v,s}(dx)\,dv \right|^2 du \notag \\
		& \le 2\int_I \Emb \left| \int_I \int_{\Rmb^d} \btil_m(X_{\frac{\lceil nu \rceil}{n}}(s),x) \left[ G_n(u,v) - G(u,v) \right] \mu_{\frac{\lceil nv \rceil}{n},s}(dx)\,dv \right|^2 du \notag \\
		& \quad + 2\int_I \Emb \left| \int_I \int_{\Rmb^d} \btil_m(X_{\frac{\lceil nu \rceil}{n}}(s),x)G(u,v)\,\mu_{\frac{\lceil nv \rceil}{n},s}(dx)\,dv \right. \notag \\
		& \qquad \left. - \int_I \int_{\Rmb^d} \btil_m(X_{\frac{\lceil nu \rceil}{n}}(s),x)G(\frac{\lceil nu \rceil}{n},v)\,\mu_{v,s}(dx)\,dv \right|^2 du \notag \\
		& =: \Tmc_s^{n,3,6} + \Tmc_s^{n,3,7}. \label{eq:dense-T35}
	\end{align}
	For $\Tmc_s^{n,3,6}$, using the definition of the bounded function $\btil_m$ in \eqref{eq:bm} we have
	\begin{align*}
		\Tmc_s^{n,3,6} & \le \kappa(M) \sum_{k=1}^m \int_I \left| \int_I \left[ G_n(u,v) - G(u,v) \right] \left[ \int_{\Rmb^d} c_k(x)\one_{\{\|x\|\le M\}} \mu_{\frac{\lceil nv \rceil}{n},s}(dx) \right] dv \right| du \\
		& \le \kappa(M) \|G_n-G\|,
	\end{align*}
	where $\kappa(M)$ depends on $M$ but not on $n$.
	For $\Tmc_s^{n,3,7}$, using Condition \ref{cond:limit-G} and Theorem \ref{thm:limit-system-property}(a) we see that 
	{
	$$\int_{\Rmb^d} \btil_m(X_{\frac{\lceil nu \rceil}{n}}(s),x)G(u,v)\,\mu_{\frac{\lceil nv \rceil}{n},s}(dx) - \int_{\Rmb^d} \btil_m(X_{\frac{\lceil nu \rceil}{n}}(s),x)G(\frac{\lceil nu \rceil}{n},v)\,\mu_{v,s}(dx) \to 0$$
	as $n \to \infty$, for each $u \in I$ and $v \in I \setminus A_u$.
	Since the set $A_u$ is assumed to have Lebesgue measure zero in Condition \ref{cond:limit-G}(b), we have}
	\begin{equation}
		\lim_{n \to \infty} \int_0^T \Tmc_s^{n,3,7} \, ds = 0. \label{eq:dense-T37}
	\end{equation}
	Combining \eqref{eq:dense-T30}--\eqref{eq:dense-T37} gives
	\begin{equation}
		\label{eq:dense-T3}
		\limsup_{n \to \infty} \int_0^T \Tmc_s^{n,3} \, ds \le \kappa(M)\limsup_{n \to \infty}\|G_n-G\|+\frac{\kappa}{M^\varepsilon}.
	\end{equation}
	
	Combining \eqref{eq:dense1}--\eqref{eq:dense-T2} gives
	\begin{equation*}
		\frac{1}{n} \sum_{i=1}^n \Emb \|X_i^n-X_{\frac{i}{n}}\|_{*,t}^2 \le \kappa \int_0^t \frac{1}{n} \sum_{i=1}^n \Emb \|X_i^n-X_{\frac{i}{n}}\|_{*,s}^2 \,ds + \frac{\kappa}{n} + \kappa \int_0^t \Tmc_s^{n,3} \, ds.
	\end{equation*}
	Using the Gronwall's inequality and \eqref{eq:dense-T3} completes the proof.
\end{proof}

\begin{Remark}
	\label{rmk:epsilon-2}
	The assumption that $\varepsilon>0$ is used in \eqref{eq:dense-T31} and \eqref{eq:dense-T32}. 
	If $b$ and $\sigma$ are bounded functions, then one can allow $\varepsilon=0$ and replace the estimate in \eqref{eq:dense-T31} and \eqref{eq:dense-T32} by 
	\begin{align*}
		\Tmc_s^{n,3,1} & \le  \frac{\kappa}{n} \sum_{i=1}^n \Emb \left| \frac{1}{n} \sum_{j=1}^n \int_{\Rmb^d} [\one_{\{|X_{\frac{i}{n}}(s)|>M\}}+\one_{\{|x|>M\}}]\,\mu_{\frac{j}{n},s}(dx) \right|^2 \le \frac{\kappa}{M^2}, \\
		\Tmc_s^{n,3,2} & \le  \frac{\kappa}{n} \sum_{i=1}^n \Emb \left| \int_I \int_{\Rmb^d} [\one_{\{|X_{\frac{i}{n}}(s)|>M\}}+\one_{\{|x|>M\}}]\,\mu_{v,s}(dx)\,dv \right|^2 \le \frac{\kappa}{M^2}.
	\end{align*}
	As a consequence, the result in Lemma \ref{lem:new1} is rewritten as $$\limsup_{n \to \infty} \frac{1}{n} \sum_{i=1}^n \Emb \|X_i^n-X_{\frac{i}{n}}\|_{*,T}^2 \le \kappa(M) \limsup_{n \to \infty} \|G_n-G\| + \frac{\kappa}{M^2}.$$
\end{Remark}

\begin{Lemma}
	\label{lem:new2}
	Suppose Conditions \ref{cond:standing} and \ref{cond:limit-G} hold.
	Let $\mubar^n := \frac{1}{n} \sum_{i=1}^n \delta_{X_{\frac{i}{n}}}$.
	Then $\mubar^n \to \mubar$ in $\Pmc(\Cmc_d)$ in probability as $n \to \infty$.
\end{Lemma}

\begin{proof}
	Let
	$\mutil^n := \frac{1}{n} \sum_{i=1}^n \mu_{\frac{i}{n}}.$	
	For each bounded and continuous function $f$ on $\Cmc_d$, using the independence of $\{X_{\frac{i}{n}}\}$, we have
	\begin{align*}
		\Emb \left( \lan f,\mubar^n \ran -\lan f,\mutil^n \ran \right)^2 & = \Emb \left( \frac{1}{n} \sum_{i=1}^n \left( f(X_{\frac{i}{n}}) - \Emb f(X_{\frac{i}{n}}) \right) \right)^2 \le \frac{\|f\|_\infty^2}{n} \to 0.
	\end{align*}
	It then follows from Condition \ref{cond:limit-G} and Theorem \ref{thm:limit-system-property}(a) that
	\begin{equation*}
	\lan f,\mutil^n \ran -\lan f,\mubar \ran = \frac{1}{n} \sum_{i=1}^n \lan f, \mu_{\frac{i}{n}} \ran - \int_I \lan f, \mu_u \ran \, du \to 0.
	\end{equation*}
	Combining these two estimates completes the proof. 
\end{proof}

\subsection{Proof of Theorem \ref{thm:dense-LLN}}
\label{sec:dense-LLN-pf}

	If Condition \ref{cond:limit-G}(b) holds, then \eqref{eq:dense-X-cvg} follows on a direct application of Lemma \ref{lem:new1} by sending $M \to \infty$.
	
	Next we prove \eqref{eq:dense-mu-cvg} by an approximation argument that is commonly used to analyze graphons (see e.g.\ \cite{BollobasJansonRiordan2007phase}).
	Fix $\eta \in (0,1)$. 
	Since $0 \le G \le 1$, using Lusin's theorem (\cite[Theorem 2.24]{Rudin1987real}), we can approximate $G$ by continuous functions $\Gtil$ with $0 \le \Gtil \le 1$.
	Namely, we can find some continuous graphon $\Gtil = \Gtil_\eta$ such that $$\|\Gtil-G\| < \eta.$$
	It then follows from Theorem \ref{thm:limit-system-property}(c) that 
	\begin{equation}
		\label{eq:dense-LLN-pf-0}
		{W_{2,T}}(\mubar^{\Gtil},\mubar) \le \kappa \eta,
	\end{equation}
	{where we denote by $X^\Gtil = \{X_u^\Gtil : u \in I\}$} the solution of \eqref{eq:limit-system-dense} corresponding to the graphon $\Gtil$ and let $\mubar^\Gtil := \int_I \Lmc(X^\Gtil_u) \, du$.
	Since Condition \ref{cond:limit-G} holds for $\Gtil$, from Lemma \ref{lem:new2} we have 
	\begin{equation}
		\label{eq:dense-LLN-pf-1}
		\frac{1}{n} \sum_{i=1}^n \delta_{X_{\frac{i}{n}}^\Gtil} \to \mubar^\Gtil
	\end{equation}
	in $\Pmc(\Cmc_d)$ in probability as $n \to \infty$.
	Note that
	\begin{equation*}
		[{W_{2,T}}(\mu^n,\frac{1}{n} \sum_{i=1}^n \delta_{X_{\frac{i}{n}}^\Gtil})]^2 \le \frac{1}{n} \sum_{i=1}^n \|X_i^n-X_{\frac{i}{n}}^\Gtil\|_{*,T}^2.
	\end{equation*}
	Since Condition \ref{cond:limit-G} holds for $\Gtil$, it follows from Lemma \ref{lem:new1} and Remark \ref{rmk:graphon-convergence} that
	\begin{equation*}
		\limsup_{n \to \infty} \Emb {W_{2,T}}(\mu^n,\frac{1}{n} \sum_{i=1}^n \delta_{X_{\frac{i}{n}}^\Gtil}) \le \sqrt{\kappa(M) \limsup_{n \to \infty} \|G_n-\Gtil\| + \frac{\kappa }{M^\varepsilon}} \le \sqrt{\kappa(M) \eta + \frac{\kappa }{M^\varepsilon}}.
	\end{equation*}
	{
	Taking $\eta \to 0$ and then $M \to \infty$ gives
	\begin{equation}
		\label{eq:dense-LLN-pf-2}
		\limsup_{\eta \to 0} \limsup_{n \to \infty} \Emb {W_{2,T}}(\mu^n,\frac{1}{n} \sum_{i=1}^n \delta_{X_{\frac{i}{n}}^\Gtil}) = 0.
	\end{equation}	
	Combining \eqref{eq:dense-LLN-pf-0}--\eqref{eq:dense-LLN-pf-2}, taking $n \to \infty$ and then sending $\eta \to 0$, we have \eqref{eq:dense-mu-cvg}.}	
	\qed	

\subsection{Proof of Theorem \ref{thm:dense-LLN-special}}
\label{sec:dense-LLN-special-pf}

	Fix $t \in [0,T]$ and $i \in \{1,\dotsc,n\}$.
	Similar to the proof of Lemma \ref{lem:new1}, we have
	\begin{align}
		& \Emb \|X_i^n-X_{\frac{i}{n}}\|_{*,t}^2 \notag \\
		& \le \kappa \Emb \int_0^t \left| \frac{1}{n} \sum_{j=1}^n \xi_{ij}^n b(X_i^n(s),X_j^n(s)) - \int_I \int_{\Rmb^d} b(X_{\frac{i}{n}}(s),x)G(\frac{i}{n},v)\,\mu_{v,s}(dx)\,dv \right|^2 ds \notag \\
		& \quad + \kappa \Emb \int_0^t \left| \frac{1}{n} \sum_{j=1}^n \xi_{ij}^n \sigma(X_i^n(s),X_j^n(s)) - \int_I \int_{\Rmb^d} \sigma(X_{\frac{i}{n}}(s),x)G(\frac{i}{n},v)\,\mu_{v,s}(dx)\,dv \right|^2 ds. \label{eq:dense5}
	\end{align}
	We will analyze the first integrand above for fixed $s \in [0,t]$, and the analysis for $\sigma$ is similar.
	By adding and subtracting terms, we have
	\begin{align}
		& \Emb \left| \frac{1}{n} \sum_{j=1}^n \xi_{ij}^n b(X_i^n(s),X_j^n(s)) - \int_I \int_{\Rmb^d} b(X_{\frac{i}{n}}(s),x)G(\frac{i}{n},v)\,\mu_{v,s}(dx)\,dv \right|^2 \notag \\
		& \le 3 \Emb \left| \frac{1}{n} \sum_{j=1}^n \xi_{ij}^n \left( b(X_i^n(s),X_j^n(s)) - b(X_{\frac{i}{n}}(s),X_{\frac{j}{n}}(s)) \right) \right|^2 \notag \\
		& \quad + 3 \Emb \left| \frac{1}{n} \sum_{j=1}^n \left( \xi_{ij}^n b(X_{\frac{i}{n}}(s),X_{\frac{j}{n}}(s)) - \int_{\Rmb^d} b(X_{\frac{i}{n}}(s),x)G(\frac{i}{n},\frac{j}{n})\,\mu_{\frac{j}{n},s}(dx)\right) \right|^2 \notag \\
		& \quad + 3 \Emb \left| \frac{1}{n} \sum_{j=1}^n \int_{\Rmb^d} b(X_{\frac{i}{n}}(s),x)G(\frac{i}{n},\frac{j}{n})\,\mu_{\frac{j}{n},s}(dx) - \int_I \int_{\Rmb^d} b(X_{\frac{i}{n}}(s),x)G(\frac{i}{n},v)\,\mu_{v,s}(dx)\,dv \right|^2 \notag \\
		& =: 3 \left( \Tmctil_s^{n,1} + \Tmctil_s^{n,2} + \Tmctil_s^{n,3} \right). \label{eq:dense6}
	\end{align}
	For $\Tmctil^{n,1}_s$, using the Lipschitz property of $b$ we have
	\begin{align}
		\Tmctil_s^{n,1} 
		& \le \kappa \Emb \left[ \frac{1}{n} \sum_{j=1}^n \left( |X_i^n(s)-X_{\frac{i}{n}}(s)|^2 + |X_j^n(s) - X_{\frac{j}{n}}(s)|^2 \right) \right]
		\le 2\kappa \max_{i=1,\dotsc,n} \Emb |X_i^n(s) - X_{\frac{i}{n}}(s)|^2. \label{eq:dense7}
	\end{align}
	For $\Tmctil^{n,2}_s$, using Condition \ref{cond:dense-Gn-special}, the independence of $\{X_{\frac{i}{n}}\}$ and $\{\xi_{ij}^n\}$, the Lipschitz property of $b$, Proposition \ref{prop:limit-existence-uniqueness}, and a weak LLN type argument, we have
	\begin{align}
		\Tmctil_s^{n,2} & \le \frac{\kappa}{n}. \label{eq:dense8}
	\end{align}
	For $\Tmctil^{n,3}_s$, we have 
	\begin{align}
		\Tmctil^{n,3}_s & = \Emb \left| \int_I \int_{\Rmb^d} b(X_{\frac{i}{n}}(s),x)G(\frac{i}{n},\frac{\lceil nv \rceil}{n})\,\mu_{\frac{\lceil nv \rceil}{n},s}(dx) \,dv - \int_I \int_{\Rmb^d} b(X_{\frac{i}{n}}(s),x)G(\frac{i}{n},v)\,\mu_{v,s}(dx)\,dv \right|^2 \notag \\
		& \le 2 \Emb \left| \int_I \int_{\Rmb^d} b(X_{\frac{i}{n}}(s),x) \left[ G(\frac{i}{n},\frac{\lceil nv \rceil}{n}) - G(\frac{i}{n},v) \right] \mu_{\frac{\lceil nv \rceil}{n},s}(dx)\,dv \right|^2 \notag \\
		& \quad + 2 \Emb \left| \int_I \left[ \int_{\Rmb^d} b(X_{\frac{i}{n}}(s),x)\,\mu_{\frac{\lceil nv \rceil}{n},s}(dx) \,dv - \int_{\Rmb^d} b(X_{\frac{i}{n}}(s),x)\,\mu_{v,s}(dx)\right]G(\frac{i}{n},v)\,dv \right|^2 \notag \\
		& \le \frac{\kappa}{n^2}, \label{eq:dense9}
	\end{align}
	where the last inequality uses Condition \ref{cond:limit-G-special}, Theorem \ref{thm:limit-system-property}(b) and Remark \ref{rmk:Wass-dual}.
	Combining \eqref{eq:dense5}--\eqref{eq:dense9} with Gronwall's inequality gives \eqref{eq:dense-special-X-cvg} and completes the proof of Theorem \ref{thm:dense-LLN-special}. \qed

\section{Proofs for Section \ref{sec:sparse-system}}
\label{sec:sparse-system-pf}

In this section we prove Theorems \ref{thm:sparse-LLN} and \ref{thm:sparse-LLN-special}.

\subsection{Preliminary Estimates}

We first provide some preliminary estimates.


\begin{Lemma}
	\label{lem:sparse-moment}
	Suppose Condition \ref{cond:standing-sparse} holds.
	Suppose either Condition \ref{cond:sparse-Gn} or Condition \ref{cond:sparse-Gn-special} holds. 
	Then
	$$\sup_{n \in \Nmb} \max_{i=1,\dotsc,n} \Emb \| X_i^n-X_{\frac{i}{n}} \|_{*,T}^k < \infty, \quad \forall k \in \Nmb.$$
\end{Lemma}

\begin{proof}[Proof of Lemma \ref{lem:sparse-moment}]
	Fix $k \in \Nmb$ and $i \in \{1,\dotsc,n\}$.
	Since $b$ and $\sigma$ are bounded, we have
	\begin{align*}
		\Emb \|X_i^n-X_{\frac{i}{n}}\|_{*,T}^k & \le \kappa \Emb \left| \frac{1}{n\beta_n} \sum_{j=1}^n \xi_{ij}^n \right|^k + \kappa \\
		& \le \kappa \frac{\sum_{j=1}^n \Emb [(\xi_{ij}^n)^k] + \left(\sum_{j=1}^n E[(\xi_{ij}^n)^2]\right)^{k/2}}{(n\beta_n)^k} + \kappa \le \kappa,
	\end{align*}
	where the second inequality follows from the Rosenthal's inequality \cite[Theorem 3]{Rosenthal1970subspaces} and the last inequality uses the assumption $\lim_{n \to \infty} n\beta_n = \infty$. 
\end{proof}

\begin{Lemma}
	\label{lem:new-sparse}
	Suppose Conditions \ref{cond:limit-G}, \ref{cond:standing-sparse} and \ref{cond:sparse-Gn}(a,b) hold.
	Then there exist some $\kappa, \kappa(M) \in (0,\infty)$ for each $M \in (0,\infty)$ such that
	\begin{equation*}
		\limsup_{n \to \infty} \frac{1}{n} \sum_{i=1}^n \Emb \|X_i^n-X_{\frac{i}{n}}\|_{*,T}^2 \le \kappa(M) \limsup_{n \to \infty} \|G_n-G\| + \frac{\kappa}{M^2}.
	\end{equation*}	
\end{Lemma}

\begin{proof}
	Fix $t \in [0,T]$.
	\begin{align}
		& \frac{1}{n} \sum_{i=1}^n \Emb \|X_i^n-X_{\frac{i}{n}}\|_{*,t}^2 \notag \\
		& \le \kappa \int_0^t \left[ \frac{1}{n} \sum_{i=1}^n \Emb \left| \frac{1}{n\beta_n} \sum_{j=1}^n \xi_{ij}^n b(X_i^n(s),X_j^n(s)) - \int_I \int_{\Rmb^d} b(X_{\frac{i}{n}}(s),x)G(\frac{i}{n},v)\,\mu_{v,s}(dx)\,dv \right|^2 \right] ds \notag \\
		& \quad + \kappa \int_0^t \left[ \frac{1}{n} \sum_{i=1}^n \Emb \left| X_i^n(s) - X_{\frac{i}{n}}(s) \right|^2 \right] ds. \label{eq:sparse1}
	\end{align}
	We will analyze the first integrand above for fixed $s \in [0,t]$.
	By adding and subtracting terms, we have
	\begin{align}
		& \frac{1}{n} \sum_{i=1}^n \Emb \left| \frac{1}{n\beta_n} \sum_{j=1}^n \xi_{ij}^n b(X_i^n(s),X_j^n(s)) - \int_I \int_{\Rmb^d} b(X_{\frac{i}{n}}(s),x)G(\frac{i}{n},v)\,\mu_{v,s}(dx)\,dv \right|^2 \notag \\
		& \le \frac{4}{n} \sum_{i=1}^n \Emb \left| \frac{1}{n} \sum_{j=1}^n \frac{\xi_{ij}^n}{\beta_n} \left( b(X_i^n(s),X_j^n(s)) - b(X_{\frac{i}{n}}(s),X_j^n(s)) \right) \right|^2 \notag \\
		& \quad + \frac{4}{n} \sum_{i=1}^n \Emb \left| \frac{1}{n} \sum_{j=1}^n \frac{\xi_{ij}^n}{\beta_n} \left( b(X_{\frac{i}{n}}(s),X_j^n(s)) - b(X_{\frac{i}{n}}(s),X_{\frac{j}{n}}(s)) \right) \right|^2 \notag \\
		& \quad + \frac{4}{n} \sum_{i=1}^n \Emb \left| \frac{1}{n} \sum_{j=1}^n \left( \frac{\xi_{ij}^n}{\beta_n} b(X_{\frac{i}{n}}(s),X_{\frac{j}{n}}(s)) - \int_{\Rmb^d} b(X_{\frac{i}{n}}(s),x)G_n(\frac{i}{n},\frac{j}{n})\,\mu_{\frac{j}{n},s}(dx)\right) \right|^2 \notag \\
		& \quad + \frac{4}{n} \sum_{i=1}^n \Emb \left| \frac{1}{n} \sum_{j=1}^n \int_{\Rmb^d} b(X_{\frac{i}{n}}(s),x)G_n(\frac{i}{n},\frac{j}{n})\,\mu_{\frac{j}{n},s}(dx) - \int_I \int_{\Rmb^d} b(X_{\frac{i}{n}}(s),x)G(\frac{i}{n},v)\,\mu_{v,s}(dx)\,dv \right|^2 \notag \\
		& =: 4 \left( \Rmc_s^{n,1} + \Rmc_s^{n,2} + \Rmc_s^{n,3} + \Rmc_s^{n,4} \right). \label{eq:sparse2}
	\end{align}
	
	Fix $s \in [0,T]$.	
	For $\Rmc_s^{n,1}$, using the Lipschitz property of $b$ we have
	\begin{align*}
		\Rmc_s^{n,1} 
		& \le \kappa \frac{1}{n} \sum_{i=1}^n \Emb \left( \frac{1}{n} \sum_{j=1}^n \frac{\xi_{ij}^n}{\beta_n} |X_i^n(s)-X_{\frac{i}{n}}(s)| \right)^2 \\
		& \le \kappa \frac{1}{n} \sum_{i=1}^n \Emb \left( \frac{1}{n} \sum_{j=1}^n \frac{\xi_{ij}^n-\beta_nG_n(\frac{i}{n},\frac{j}{n})}{\beta_n} |X_i^n(s)-X_{\frac{i}{n}}(s)| \right)^2 \\
		& \quad + \kappa \frac{1}{n} \sum_{i=1}^n \Emb \left( \frac{1}{n} \sum_{j=1}^n G_n(\frac{i}{n},\frac{j}{n}) |X_i^n(s)-X_{\frac{i}{n}}(s)| \right)^2.
	\end{align*}
	Since $\{\xi_{ij}^n-\beta_nG_n(\frac{i}{n},\frac{j}{n})\}$ are centered and independent, we have
	\begin{align*}
		& \Emb \left( \frac{1}{n} \sum_{j=1}^n \frac{\xi_{ij}^n-\beta_nG_n(\frac{i}{n},\frac{j}{n})}{\beta_n} \right)^4 \\
		& = \frac{1}{(n\beta_n)^4} \sum_{j=1}^n \Emb \left(\xi_{ij}^n-\beta_nG_n(\frac{i}{n},\frac{j}{n})\right)^4 \\
		& \quad + \frac{3}{(n\beta_n)^4} \sum_{j=1}^n \sum_{k \ne j}^n \Emb \left(\xi_{ij}^n-\beta_nG_n(\frac{i}{n},\frac{j}{n})\right)^2 \Emb \left(\xi_{ik}^n-\beta_nG_n(\frac{i}{n},\frac{k}{n})\right)^2 \\
		& \le \frac{1}{(n\beta_n)^3} + \frac{3}{(n\beta_n)^2}.
	\end{align*}
	From these two estimates, Cauchy--Schwarz inequality and Lemma \ref{lem:sparse-moment} we have 
	\begin{equation}
		\Rmc_s^{n,1} \le \frac{\kappa}{n\beta_n} + \frac{\kappa}{n} \sum_{i=1}^n \Emb |X_i^n(s)-X_{\frac{i}{n}}(s)|^2. \label{eq:sparse-R1}
	\end{equation}
	
	For $\Rmc_s^{n,3}$, using Condition \ref{cond:sparse-Gn}, the independence of $\{X_{\frac{i}{n}}\}$ and $\{\xi_{ij}^n\}$, the boundedness of $b$, and a weak LLN type argument, we have 
	\begin{equation}
		\Rmc_s^{n,3} \le \frac{\kappa}{n\beta_n}. \label{eq:sparse-R3}
	\end{equation}
	
	For $\Rmc_s^{n,4}$, similar to the proof of \eqref{eq:dense-T3} in Lemma \ref{lem:new1} and Remark \ref{rmk:epsilon-2}, using Conditions \ref{cond:limit-G}, \ref{cond:standing-sparse} and \ref{cond:sparse-Gn}(a,b) we could get
	\begin{equation}
		\limsup_{n \to \infty} \int_0^T \Rmc_s^{n,4} \, ds \le \kappa(M)\limsup_{n \to \infty}\|G_n-G\|+\frac{\kappa}{M^2}. \label{eq:sparse-R4}
	\end{equation}

	The analysis of $\Rmc_s^{n,2}$ is based on a collection of change of measure arguments.
	First note that
	\begin{align}
		\Rmc_s^{n,2} & \le \kappa \frac{1}{n} \sum_{i=1}^n \Emb \left( \frac{1}{n} \sum_{j=1}^n \frac{\xi_{ij}^n}{\beta_n} |X_j^n(s)-X_{\frac{j}{n}}(s)| \right)^2 \notag \\
		& = \kappa \frac{1}{n} \sum_{i=1}^n \frac{1}{(n\beta_n)^2} \left( \sum_{j=1}^n \Emb \left[ \xi_{ij}^n |X_j^n(s)-X_{\frac{j}{n}}(s)|^2 \right] \right. \notag \\
		& \left. \quad + \sum_{j=1}^n \sum_{k \ne j}^n \Emb \left[ \xi_{ij}^n \xi_{ik}^n |X_j^n(s)-X_{\frac{j}{n}}(s)| |X_k^n(s)-X_{\frac{k}{n}}(s)| \right] \right). \label{eq:Tn2-state-dependent}
	\end{align}	
	Fix $i,j \in \{1,\dotsc,n\}$.
	We want to show that $\Emb \left[ \xi_{ij}^n (X_j^n(s)-X_{\frac{j}{n}}(s))^2 \right]$ is roughly $\Emb \left[ (X_j^n(s)-X_{\frac{j}{n}}(s))^2 \right] \beta_n G_n(\frac{i}{n},\frac{j}{n})$ up to some errors, namely $\xi_{ij}^n$ is kind of independent of $X_j^n(s)-X_{\frac{j}{n}}(s)$.
	For this, we want to compare $\Lmc(X_j^n,X_{\frac{j}{n}} \,|\, \xi_{ij}^n=1)$ and $\Lmc(X_j^n,X_{\frac{j}{n}})$ using Girsanov's theorem.
	This heuristic suggests considering the following auxiliary processes
	\begin{align*}
		\Xtil_{\frac{i}{n}}(t) & = X_{\frac{i}{n}}(0) + \int_0^t \int_I \int_{\Rmb^d} b(\Xtil_{\frac{i}{n}}(s),x)G(\frac{i}{n},v)\,\mu_{v,s}(dx)\,dv\,ds + \int_0^t \sigma(\Xtil_{\frac{i}{n}}(s))\,dB_{\frac{i}{n}}(s) \\
		& \quad + \int_0^t \frac{\sigma(\Xtil_{\frac{i}{n}}(s))}{\sigma(X_i^n(s))} \frac{1}{n\beta_n} \left( \xi_{ij}^n - 1 \right) b(X_i^n(s),X_j^n(s)) \, ds, \\
		\Xtil_{\frac{j}{n}}(t) & = X_{\frac{j}{n}}(0) + \int_0^t \int_I \int_{\Rmb^d} b(\Xtil_{\frac{j}{n}}(s),x)G(\frac{j}{n},v)\,\mu_{v,s}(dx)\,dv\,ds + \int_0^t \sigma(\Xtil_{\frac{j}{n}}(s))\,dB_{\frac{j}{n}}(s) \\
		& \quad + \int_0^t \frac{\sigma(\Xtil_{\frac{j}{n}}(s))}{\sigma(X_j^n(s))} \frac{1}{n\beta_n} \left( \xi_{ji}^n - 1 \right) b(X_j^n(s),X_i^n(s)) \, ds.
	\end{align*}
	Note that the existence and uniqueness of such processes are guaranteed by the bounded and Lipschitz properties of $b$, $\sigma$ and $\sigma^{-1}$.
	Also using these properties and Gronwall's inequality we can show that
	\begin{equation}
		\label{eq:moment-auxiliary}
		\Emb \|X_{\frac{j}{n}}-\Xtil_{\frac{j}{n}}\|_{*,T}^m \le \frac{\kappa(m)}{(n\beta_n)^m}, \quad m \ge 0.
	\end{equation}
	Define $Q^{i,j,n}$ by 
	\begin{align*}
		\frac{dQ^{i,j,n}}{d\Pmb} & = \Emc_T \left( \int_0^\cdot \frac{1}{\sigma(X_i^n(s))} \frac{1}{n\beta_n} \left( 1 - \xi_{ij}^n \right) b(X_i^n(s),X_j^n(s)) \, dB_{\frac{i}{n}}(s) \right. \\
		& \quad \left. + \int_0^\cdot \frac{1}{\sigma(X_j^n(s))} \frac{1}{n\beta_n} \left( 1 - \xi_{ji}^n \right) b(X_j^n(s),X_i^n(s)) \, dB_{\frac{j}{n}}(s) \right),
	\end{align*}
	where
	$$\Emc_t(M) := \exp \left\{ M_t - \frac{1}{2} [M]_t \right\}$$
	is the Doleans exponential for a semi-martingale $M_t$.
	Since $b$ and $\sigma^{-1}$ are bounded, we have that $B_{\frac{i}{n}}(\cdot) - \int_0^\cdot \frac{1}{\sigma(X_i^n(s))} \frac{1}{n\beta_n} ( 1 - \xi_{ij}^n ) b(X_i^n(s),X_j^n(s)) \, ds$, $B_{\frac{j}{n}}(\cdot) - \int_0^\cdot \frac{1}{\sigma(X_j^n(s))} \frac{1}{n\beta_n} ( 1 - \xi_{ji}^n ) b(X_j^n(s),X_i^n(s)) \, ds$ and $B_{\frac{k}{n}}$, $k \ne i,j$, are independent Brownian motion under $Q^{i,j,n}$,
	and
	\begin{equation}
		\label{eq:Girsanov-0}
		\Pmb\left( \left(X_i^n,X_j^n,X_{\frac{i}{n}},X_{\frac{j}{n}}\right) \in \cdot \,|\, \xi_{ij}^n=1 \right) = Q^{i,j,n}\left( \left(X_i^n,X_j^n,\Xtil_{\frac{i}{n}},\Xtil_{\frac{j}{n}}\right) \in \cdot \right)
	\end{equation}
	by Girsanov's theorem,
	and
	\begin{equation*}
		\Emb \left[ \left( \frac{dQ^{i,j,n}}{d\Pmb} \right)^m \right] \le \exp \left\{ \frac{m|m-1|\kappa}{(n\beta_n)^2} \right\}, \quad m \ge 0.
	\end{equation*}
	From this it then follows that
	\begin{align}
		\Emb \left[ \left| \frac{dQ^{i,j,n}}{d\Pmb} - 1\right|^m \right] & \le \sqrt{ \Emb \left[ \left| \frac{dQ^{i,j,n}}{d\Pmb} - 1\right|^2 \right] \Emb \left[ \left| \frac{dQ^{i,j,n}}{d\Pmb} - 1\right|^{2m-2} \right] } \notag \\
		& \le \sqrt{ \Emb \left[ \left( \frac{dQ^{i,j,n}}{d\Pmb} \right)^2 - 1 \right] \kappa(m) \Emb \left[ \left| \frac{dQ^{i,j,n}}{d\Pmb} \right|^{2m-2} + 1 \right] } \notag \\
		& \le \sqrt{ \left[ 1 + \frac{\kappa(m)}{(n\beta_n)^2} - 1 \right] \kappa(m) } \le \frac{\kappa(m)}{n\beta_n}, \quad m \ge 1, \label{eq:Girsanov-1}
	\end{align}
	where the third inequality uses the assumption that $\lim_{n \to \infty} n\beta_n = \infty$.
	From \eqref{eq:Girsanov-0} we have
	\begin{align*}
		\Emb \left[ \xi_{ij}^n (X_j^n(s)-X_{\frac{j}{n}}(s))^2 \right] & = \Emb \left[ (X_j^n(s)-X_{\frac{j}{n}}(s))^2 \,|\, \xi_{ij}^n=1 \right] \beta_n G_n(\frac{i}{n},\frac{j}{n}) \\
		& = \Emb_{Q^{i,j,n}} \left[ (X_j^n(s)-\Xtil_{\frac{j}{n}}(s))^2 \right] \beta_n G_n(\frac{i}{n},\frac{j}{n}).
	\end{align*}
	Note that
	\begin{align*}
		\Emb_{Q^{i,j,n}} \left[ (X_j^n(s)-\Xtil_{\frac{j}{n}}(s))^2 \right] & = \Emb \left[ (X_j^n(s)-\Xtil_{\frac{j}{n}}(s))^2 \frac{dQ^{i,j,n}}{d\Pmb} \right] \\
		& \le 2\Emb \left[ (X_j^n(s)-X_{\frac{j}{n}}(s))^2 \frac{dQ^{i,j,n}}{d\Pmb} \right] + 2\Emb \left[ (X_{\frac{j}{n}}(s)-\Xtil_{\frac{j}{n}}(s))^2 \frac{dQ^{i,j,n}}{d\Pmb} \right].
	\end{align*}
	For the first term, using Holder's inequality, Lemma \ref{lem:sparse-moment} and \eqref{eq:Girsanov-1} we have
	\begin{align*}
		& \Emb \left[ (X_j^n(s)-X_{\frac{j}{n}}(s))^2 \frac{dQ^{i,j,n}}{d\Pmb} \right] - \Emb \left[ (X_j^n(s)-X_{\frac{j}{n}}(s))^2 \right] \\
		& \le \left( \Emb \left[ (X_j^n(s)-X_{\frac{j}{n}}(s))^{2p} \right] \right)^{1/p} \left( \Emb \left[ \left|\frac{dQ^{i,j,n}}{d\Pmb}-1\right|^{q} \right] \right)^{1/q} \\
		& \le \frac{\kappa(q)}{(n\beta_n)^{1/q}}, \quad \forall q > 1.
	\end{align*}
	For the second term, using Holder's inequality and \eqref{eq:moment-auxiliary} we have
	\begin{align*}
		\Emb \left[ \left(X_{\frac{j}{n}}(s)-\Xtil_{\frac{j}{n}}(s)\right)^2 \frac{dQ^{i,j,n}}{d\Pmb} \right] \le \sqrt{\Emb \left[ \left(X_{\frac{j}{n}}(s)-\Xtil_{\frac{j}{n}}(s)\right)^4 \right] \Emb \left[\left(\frac{dQ^{i,j,n}}{d\Pmb}\right)^2\right]} \le \frac{\kappa}{(n\beta_n)^2}.
	\end{align*}
	Combining these three estimates gives
	\begin{align}
		\label{eq:moment}
		\Emb \left[ \xi_{ij}^n (X_j^n(s)-X_{\frac{j}{n}}(s))^2 \right] \le \left( 2\Emb \left[ (X_j^n(s)-X_{\frac{j}{n}}(s))^2 \right] + \frac{\kappa(q)}{(n\beta_n)^{1/q}} \right) \beta_n G_n(\frac{i}{n},\frac{j}{n}).
	\end{align}
	
	Now fix $i,j,k \in \{1,\dotsc,n\}$ with $j \ne k$.
	The following argument is similar to the above change of measure, but we provide the proof for completeness.
	The intuition is again that $\Emb \left[ \xi_{ij}^n \xi_{ik}^n |X_j^n(s)-X_{\frac{j}{n}}(s)| |X_k^n(s)-X_{\frac{k}{n}}(s)| \right]$ is roughly $\Emb \left[ |X_j^n(s)-X_{\frac{j}{n}}(s)| |X_k^n(s)-X_{\frac{k}{n}}(s)| \right] \beta_n^2 G_n(\frac{i}{n},\frac{j}{n}) G_n(\frac{i}{n},\frac{k}{n})$ up to some errors, namely $(\xi_{ij}^n,\xi_{ik}^n)$ is kind of independent of $(X_j^n(s)-X_{\frac{j}{n}}(s))(X_k^n(s)-X_{\frac{k}{n}}(s))$.
	For this, we want to compare $\Lmc(X_j^n,X_{\frac{j}{n}},X_k^n,X_{\frac{k}{n}} \,|\, \xi_{ij}^n=1,\xi_{ik}^n=1)$ and $\Lmc(X_j^n,X_{\frac{j}{n}},X_k^n,X_{\frac{k}{n}})$ using Girsanov's theorem.
	This heuristic suggests considering the following auxiliary processes
	\begin{align*}
		\Xtil_{\frac{i}{n}}(t) & = X_{\frac{i}{n}}(0) + \int_0^t \int_I \int_{\Rmb^d} b(\Xtil_{\frac{i}{n}}(s),x)G(\frac{i}{n},v)\,\mu_{v,s}(dx)\,dv\,ds + \int_0^t \sigma(\Xtil_{\frac{i}{n}}(s))\,dB_{\frac{i}{n}}(s) \\
		& \quad + \sum_{l=j,k} \int_0^t \frac{\sigma(\Xtil_{\frac{i}{n}}(s))}{\sigma(X_i^n(s))} \frac{1}{n\beta_n} \left( \xi_{il}^n - 1 \right) b(X_i^n(s),X_l^n(s)) \, ds, \\
		\Xtil_{\frac{j}{n}}(t) & = X_{\frac{j}{n}}(0) + \int_0^t \int_I \int_{\Rmb^d} b(\Xtil_{\frac{j}{n}}(s),x)G(\frac{j}{n},v)\,\mu_{v,s}(dx)\,dv\,ds + \int_0^t \sigma(\Xtil_{\frac{j}{n}}(s))\,dB_{\frac{j}{n}}(s) \\
		& \quad + \int_0^t \frac{\sigma(\Xtil_{\frac{j}{n}}(s))}{\sigma(X_j^n(s))} \frac{1}{n\beta_n} \left( \xi_{ji}^n - 1 \right) b(X_j^n(s),X_i^n(s)) \, ds \\
		\Xtil_{\frac{k}{n}}(t) & = X_{\frac{k}{n}}(0) + \int_0^t \int_I \int_{\Rmb^d} b(\Xtil_{\frac{k}{n}}(s),x)G(\frac{k}{n},v)\,\mu_{v,s}(dx)\,dv\,ds + \int_0^t \sigma(\Xtil_{\frac{k}{n}}(s))\,dB_{\frac{k}{n}}(s) \\
		& \quad + \int_0^t \frac{\sigma(\Xtil_{\frac{k}{n}}(s))}{\sigma(X_k^n(s))} \frac{1}{n\beta_n} \left( \xi_{ki}^n - 1 \right) b(X_k^n(s),X_i^n(s)) \, ds.
	\end{align*}
	Note that the existence and uniqueness of such processes are again guaranteed by the bounded and Lipschitz properties of $b$, $\sigma$ and $\sigma^{-1}$.
	Also using these properties and Gronwall's inequality we can show that
	\begin{equation}
		\label{eq:moment-auxiliary-jk}
		\Emb \|X_{\frac{j}{n}}-\Xtil_{\frac{j}{n}}\|_{*,T}^m + \Emb \|X_{\frac{k}{n}}-\Xtil_{\frac{k}{n}}\|_{*,T}^m \le \frac{\kappa(m)}{(n\beta_n)^m}, \quad m \ge 0.
	\end{equation}
	Define $Q^{i,j,k,n}$ by
	\begin{align*}
		\frac{dQ^{i,j,k,n}}{d\Pmb} & = \Emc_T \left( \sum_{l=j,k} \int_0^\cdot \frac{1}{\sigma(X_i^n(s))} \frac{1}{n\beta_n} \left( 1 - \xi_{il}^n \right) b(X_i^n(s),X_l^n(s)) \, dB_{\frac{i}{n}} \right. \\
		& \quad \left. + \sum_{l=j,k} \int_0^\cdot \frac{1}{\sigma(X_l^n(s))} \frac{1}{n\beta_n} \left( 1 - \xi_{li}^n \right) b(X_l^n(s),X_i^n(s)) \, dB_{\frac{l}{n}} \right).
	\end{align*}
	Since $b$ and $\sigma^{-1}$ are bounded, we have 
	\begin{align}
		& \Pmb\left( \left(X_i^n,X_j^n,X_k^n,X_{\frac{i}{n}},X_{\frac{j}{n}},X_{\frac{k}{n}}\right) \in \cdot \,|\, \xi_{ij}^n=1, \xi_{ik}^n=1 \right) \notag \\
		& \quad = Q^{i,j,k,n}\left( \left(X_i^n,X_j^n,X_k^n,\Xtil_{\frac{i}{n}},\Xtil_{\frac{j}{n}},\Xtil_{\frac{k}{n}}\right) \in \cdot \right) \label{eq:Girsanov-jk-0}
	\end{align}
	by Girsanov's theorem,
	and
	\begin{equation*}
		\Emb \left[ \left( \frac{dQ^{i,j,k,n}}{d\Pmb} \right)^m \right] \le \exp \left\{ \frac{m|m-1|\kappa}{(n\beta_n)^2} \right\}, m \ge 0.
	\end{equation*}
	Using this and the assumption that $\lim_{n \to \infty} n\beta_n = \infty$, we have
	\begin{equation}
		\label{eq:Girsanov-jk-1}
		\Emb \left[ \left| \frac{dQ^{i,j,n}}{d\Pmb} - 1\right|^m \right] \le \sqrt{ \Emb \left[ \left| \frac{dQ^{i,j,n}}{d\Pmb} - 1\right|^2 \right] \Emb \left[ \left| \frac{dQ^{i,j,n}}{d\Pmb} - 1\right|^{2m-2} \right] } \le \frac{\kappa(m)}{n\beta_n}, m \ge 1.
	\end{equation}
	From \eqref{eq:Girsanov-jk-0} we have
	\begin{align*}
		& \Emb \left[ \xi_{ij}^n \xi_{ik}^n |X_j^n(s)-X_{\frac{j}{n}}(s)| |X_k^n(s)-X_{\frac{k}{n}}(s)| \right] \\
		& = \Emb \left[ |X_j^n(s)-X_{\frac{j}{n}}(s)| |X_k^n(s)-X_{\frac{k}{n}}(s)| \,|\, \xi_{ij}^n=1, \xi_{ik}^n=1 \right] \beta_n^2 G_n(\frac{i}{n},\frac{j}{n}) G_n(\frac{i}{n},\frac{k}{n}) \\
		& = \Emb_{Q^{i,j,k,n}} \left[ |X_j^n(s)-\Xtil_{\frac{j}{n}}(s)| |X_k^n(s)-\Xtil_{\frac{k}{n}}(s)| \right] \beta_n^2 G_n(\frac{i}{n},\frac{j}{n}) G_n(\frac{i}{n},\frac{k}{n}).
	\end{align*}
	Note that
	\begin{align*}
		& \Emb_{Q^{i,j,k,n}} \left[ |X_j^n(s)-\Xtil_{\frac{j}{n}}(s)| |X_k^n(s)-\Xtil_{\frac{k}{n}}(s)| \right] \\
		& \le \half \Emb_{Q^{i,j,k,n}} \left[ |X_j^n(s)-\Xtil_{\frac{j}{n}}(s)|^2 \right] + \half \Emb_{Q^{i,j,k,n}} \left[ |X_k^n(s)-\Xtil_{\frac{k}{n}}(s)|^2 \right] \\
		& = \half \Emb \left[ (X_j^n(s)-\Xtil_{\frac{j}{n}}(s))^2 \frac{dQ^{i,j,k,n}}{d\Pmb} \right] + \half \Emb \left[ (X_k^n(s)-\Xtil_{\frac{k}{n}}(s))^2 \frac{dQ^{i,j,k,n}}{d\Pmb} \right] \\
		& \le \Emb \left[ (X_j^n(s)-X_{\frac{j}{n}}(s))^2 \frac{dQ^{i,j,k,n}}{d\Pmb} \right] + \Emb \left[ (X_{\frac{j}{n}}(s)-\Xtil_{\frac{j}{n}}(s))^2 \frac{dQ^{i,j,k,n}}{d\Pmb} \right] \\
		& \quad + \Emb \left[ (X_k^n(s)-X_{\frac{k}{n}}(s))^2 \frac{dQ^{i,j,k,n}}{d\Pmb} \right] + \Emb \left[ (X_{\frac{k}{n}}(s)-\Xtil_{\frac{k}{n}}(s))^2 \frac{dQ^{i,j,k,n}}{d\Pmb} \right].
	\end{align*}
	Similar to the derivation of \eqref{eq:moment}, using Holder's inequality, Lemma \ref{lem:sparse-moment}, \eqref{eq:Girsanov-jk-1}, and \eqref{eq:moment-auxiliary-jk} we have
	\begin{align}
		& \Emb \left[ \xi_{ij}^n \xi_{ik}^n |X_j^n(s)-X_{\frac{j}{n}}(s)| |X_k^n(s)-X_{\frac{k}{n}}(s)| \right] \label{eq:moment-jk} \\
		& \quad \le \left( \Emb \left[ (X_j^n(s)-X_{\frac{j}{n}}(s))^2 \right] + \Emb \left[ (X_k^n(s)-X_{\frac{k}{n}}(s))^2 \right] + \frac{\kappa_q}{(n\beta_n)^{1/q}} \right) \beta_n^2 G_n(\frac{i}{n},\frac{j}{n}) G_n(\frac{i}{n},\frac{k}{n}). \notag
	\end{align}
	
	Applying \eqref{eq:moment} and \eqref{eq:moment-jk} to \eqref{eq:Tn2-state-dependent} gives
	\begin{equation}
		\Rmc_s^{n,2} \le \frac{\kappa}{n} \sum_{j=1}^n \Emb |X_j^n(s)-X_{\frac{j}{n}}(s)|^2 + \frac{\kappa(q)}{(n\beta_n)^{1/q}}. \label{eq:sparse-R2}
	\end{equation}
	
	Finally, combining \eqref{eq:sparse1}--\eqref{eq:sparse-R3} and \eqref{eq:sparse-R2} gives
	\begin{equation*}
		\frac{1}{n} \sum_{i=1}^n \Emb \|X_i^n-X_{\frac{i}{n}}\|_{*,t}^2 \le \kappa \int_0^t \frac{1}{n} \sum_{i=1}^n \Emb \|X_i^n-X_{\frac{i}{n}}\|_{*,s}^2 + \frac{\kappa(q)}{(n\beta_n)^{1/q}} + \kappa \int_0^t \Rmc_s^{n,4} \, ds.
	\end{equation*}
	Using the Gronwall's inequality and \eqref{eq:sparse-R4} completes the proof. 
\end{proof}

\subsection{Proof of Theorem \ref{thm:sparse-LLN}}
\label{sec:sparse-LLN-pf}

	The proof is similar to that of Theorem \ref{thm:dense-LLN}.
	If Condition \ref{cond:limit-G}(b) holds, then \eqref{eq:sparse-X-cvg} follows on a direct application of Lemma \ref{lem:new-sparse} by sending $M \to \infty$.
	
	Next we prove \eqref{eq:sparse-mu-cvg}.
	Fix $\eta \in (0,1)$.
	Since $0 \le G \le 1$, using Lusin's theorem (\cite[Theorem 2.24]{Rudin1987real}), we can approximate $G$ by continuous functions $\Gtil$ with $0 \le \Gtil \le 1$.
	Namely, we can find some continuous graphon $\Gtil = \Gtil_\eta$ such that $$\|\Gtil-G\| < \eta.$$
	It then follows from Theorem \ref{thm:limit-system-property}(c) that 
	\begin{equation}
		\label{eq:sparse-LLN-pf-0}
		{W_{2,T}}(\mubar^{\Gtil},\mubar) \le \kappa \eta
	\end{equation}
	{where we denote by $X^\Gtil = \{X_u^\Gtil : u \in I\}$} the solution of \eqref{eq:limit-system-sparse} corresponding to the graphon $\Gtil$ and let $\mubar^\Gtil := \int_I \Lmc(X^\Gtil_u) \, du$.
	Since Condition \ref{cond:limit-G} holds for $\Gtil$, from Lemma \ref{lem:new2} we have 
	\begin{equation}
		\label{eq:sparse-LLN-pf-1}
		\frac{1}{n} \sum_{i=1}^n \delta_{X_{\frac{i}{n}}^\Gtil} \to \mubar^\Gtil
	\end{equation}
	in $\Pmc(\Cmc_d)$ in probability as $n \to \infty$.
	Note that
	\begin{equation*}
		[{W_{2,T}}(\mu^n,\frac{1}{n} \sum_{i=1}^n \delta_{X_{\frac{i}{n}}^\Gtil})]^2 \le \frac{1}{n} \sum_{i=1}^n \|X_i^n-X_{\frac{i}{n}}^\Gtil\|_{*,T}^2.
	\end{equation*}
	Since Condition \ref{cond:limit-G} holds for $\Gtil$, it follows from Lemma \ref{lem:new-sparse} and Remark \ref{rmk:graphon-convergence} that
	\begin{equation*}
		\limsup_{n \to \infty} \Emb {W_{2,T}}(\mu^n,\frac{1}{n} \sum_{i=1}^n \delta_{X_{\frac{i}{n}}^\Gtil}) \le \sqrt{\kappa(M) \limsup_{n \to \infty} \|G_n-\Gtil\| + \frac{\kappa}{M^\varepsilon}} \le \sqrt{\kappa(M) \eta + \frac{\kappa}{M^\varepsilon}}.
	\end{equation*}
	{
	Taking $\eta \to 0$ and then $M \to \infty$ gives
	\begin{equation}
		\label{eq:sparse-LLN-pf-2}
		\limsup_{\eta \to 0} \limsup_{n \to \infty} \Emb {W_{2,T}}(\mu^n,\frac{1}{n} \sum_{i=1}^n \delta_{X_{\frac{i}{n}}^\Gtil}) = 0.
	\end{equation}	
	Combining \eqref{eq:sparse-LLN-pf-0}--\eqref{eq:sparse-LLN-pf-2}, taking $n \to \infty$ and then sending $\eta \to 0$, we have \eqref{eq:sparse-mu-cvg}.}
	\qed

\subsection{Proof of Theorem \ref{thm:sparse-LLN-special}}

	Note that the estimates in \eqref{eq:sparse1}--\eqref{eq:sparse-R3} and \eqref{eq:sparse-R2} still hold under Conditions \ref{cond:limit-G-special}, \ref{cond:standing-sparse} and \ref{cond:sparse-Gn-special}. 
	In order to show \eqref{eq:sparse-special-X-cvg}, it suffices to argue
	$$\Rmc^{n,4}_s \le \frac{\kappa}{n}.$$	
	For this, using Condition \ref{cond:sparse-Gn-special} we have
	\begin{align*}
		\Rmc^{n,4}_s & = \frac{1}{n} \sum_{i=1}^n \Emb \left| \frac{1}{n} \sum_{j=1}^n \int_{\Rmb^d} b(X_{\frac{i}{n}}(s),x)G(\frac{i}{n},\frac{j}{n})\,\mu_{\frac{j}{n},s}(dx) \right. \\
		& \left. \qquad - \int_I \int_{\Rmb^d} b(X_{\frac{i}{n}}(s),x)G(\frac{i}{n},v)\,\mu_{v,s}(dx)\,dv \right|^2\\
		& = \frac{1}{n} \sum_{i=1}^n \Emb \left| \int_I \int_{\Rmb^d} b(X_{\frac{i}{n}}(s),x)G(\frac{i}{n},\frac{\lceil nv \rceil}{n})\,\mu_{\frac{\lceil nv \rceil}{n},s}(dx) \,dv \right. \\
		& \left. \qquad - \int_I \int_{\Rmb^d} b(X_{\frac{i}{n}}(s),x)G(\frac{i}{n},v)\,\mu_{v,s}(dx)\,dv \right|^2 \\
		& \le 2 \frac{1}{n} \sum_{i=1}^n \Emb \left| \int_I \int_{\Rmb^d} b(X_{\frac{i}{n}}(s),x) \left[ G(\frac{i}{n},\frac{\lceil nv \rceil}{n}) - G(\frac{i}{n},v) \right] \mu_{\frac{\lceil nv \rceil}{n},s}(dx)\,dv \right|^2 \\
		& \quad + 2 \frac{1}{n} \sum_{i=1}^n \Emb \left| \int_I \left[ \int_{\Rmb^d} b(X_{\frac{i}{n}}(s),x)\,\mu_{\frac{\lceil nv \rceil}{n},s}(dx) - \int_{\Rmb^d} b(X_{\frac{i}{n}}(s),x)\,\mu_{v,s}(dx)\right]G(\frac{i}{n},v)\,dv \right|^2 \\
		& \le \frac{\kappa}{n^2},
	\end{align*}
	where the last inequality uses Condition \ref{cond:limit-G-special}, Theorem \ref{thm:limit-system-property}(b) and Remark \ref{rmk:Wass-dual}.
	This completes the proof of Theorem \ref{thm:sparse-LLN-special}. 
	\qed

\section{Acknowledgment} 

{We would like to thank the anonymous referees for their careful reading and many valuable suggestions on the paper.}

\bibliographystyle{plain}


\end{document}